\documentclass[11pt]{amsart}
\usepackage{amsfonts,amssymb,amscd}
\usepackage{oldgerm}
\usepackage{amsfonts}
\usepackage{amssymb}
\usepackage{graphics}
\usepackage{booktabs}

\usepackage{booktabs, tabularx}
\usepackage[margin=0.8in]{geometry}

\usepackage{xfrac}

\usepackage[english]{babel}
\usepackage[T1]{fontenc}
\usepackage{ucs}
\usepackage[utf8x]{inputenc}
\usepackage{booktabs}

\usepackage{xfrac}
\usepackage{epic,eepic,pstricks,graphics,color}

\usepackage{ucs}
\usepackage[utf8x]{inputenc}

\usepackage{epic,eepic,pstricks,graphics,color}

\newcommand{\C}{\mathbb{C}}
\newcommand{\R}{\mathbb{R}}
\newcommand{\N}{\mathbb{N}}
\newcommand{\Z}{\mathbb{Z}}

\newcommand{\mcf}{\mathcal{D}}

\newcommand{\fol}{\mathcal{F}}

\newcommand{\calp}{{\mathcal{P}}}

\def\picill#1by#2(#3)#4
{\vbox to #2
{\hrule width #1 height 0pt depth 0pt
\vfill\special{illustration #3 scaled #4}}}

\newtheorem{teo}{Theorem}[section]
\newtheorem{prop}[teo]{Proposition}
\newtheorem{lema}[teo]{Lemma}
\newtheorem{obs}[teo]{Remark}
\newtheorem{defnc}[teo]{Definition}

\date{\today}

\begin{document}

\title{Palais leaf-space manifolds and surfaces carrying holomorphic flows}

\author{Ana Cristina Ferreira, \, \, \, Julio C. Rebelo \, \, \, \& \, \, \, Helena Reis}
\address{}
\thanks{}


\begin{abstract}
Given a pair of commuting holomorphic vector fields defined on a neighborhood of $(0,0) \in \C^2$, we discuss the
problem of globalizing them as an action of $\C^2$ on a suitable complex surfaces along with some related questions.
A review of Palais' theory about globalization of local transformation groups is also included in our discussion.
\end{abstract}

\dedicatory{To Professor Yu.S. Il'yashenko, on the occasion of his ${\rm 75}^{\rm th}$ birthday}

\maketitle

\section{Introduction}

This paper is a first step towards a project about finding new {\it transcendent}\,
complete vector fields and/or new (transcendent) open complex surfaces (or, more generally, manifolds).
In this sense, this paper is vaguely related to the preprint \cite{FRR} though the results in \cite{FRR} have no
bear in the present discussion.

Roughly speaking, the project mentioned above is to a large extent concerned with an open complex surface
$M$ equipped either with a $\C$-action or with a $\C^2$-action having an open orbit, which of course constitutes a far more
symmetric case. The latter situation can equivalently be formulated by saying that the surface $M$ is
equipped with a pair of commuting (holomorphic) complete vector fields $X, \, Y$ that are linearly independent at
generic points. A first motivation to pay attention to vector fields as above is that, among them, there may exist
some new examples of dynamical behavior. Alternatively, their solutions might provide new (interesting)
transcendent functions in the spirit of the works of Painlev\'e and Chazy among others.

A first observation that explains our emphasis on the transcendent setting (e.g. open surfaces, non-algebraic vector fields)
stems from the main result in \cite{GuR}. Indeed, \cite{GuR} contains a rather detailed picture of the dynamics of a complete
(in fact, semicomplete) algebraic vector field on an algebraic surface. In particular, if genuinely new dynamical phenomena are
to exist among complete vector fields on surfaces, they can only be found in the ``transcendent context''.

It turns out, however, that dealing with transcendent vector fields and/or non-algebraic open surfaces is rather difficult
and the number of available tools is small compared to the corresponding situation in the algebraic setting.
For example, relatively little is known about (non-algebraic) ``entire complete vector fields'' already on $\C^2$
and this is in stark contrast with the results of \cite{GuR}. In this sense, it can be said that the
study of complete transcendent vector fields is wide open in terms of new possibilities.

Yet, to begin a systematic study of transcendent complete vector fields, we must not only face the
paucity of suitable tools but also the absence of a convenient starting point where to conduct a detailed - local
or semi-global - analysis. Again the situation contrasts with the
discussion in \cite{GuR} which started by focusing at the pole divisor of the vector field in question which provides an invariant
analytic curve for the underlying foliation. In the case of entire vector fields on $\C^2$, however, the existence of invariant
curves is far from evident. With all these issues taken into account,
the approach initiated in this paper is based on going
from {\it local}\, to {\it global}. More precisely, we will globalize a suitably chosen family of germs of
commuting vector fields on $(\C^2 ,0)$ into an action of $\C^2$ defined on a certain complex surface.

To be more accurate, consider the pairs $X$, $Y$ of commuting vector fields defined on a neighborhood of $(0,0) \in \C^2$ and
admitting the normal forms
\begin{eqnarray}
X & = &   x^a y^b\left[ mx \frac{\partial}{\partial x} - ny \frac{\partial}{\partial y} \right] \; \; \; {\rm and} \nonumber \\
Y & = & g(x^n y^m) \left[ x (-bm + x^a y^b f(x^ny^m)) \frac{\partial}{\partial x} + y \left(a m - \frac{n}{m} x^a y^b f(x^n y^m)\right)
\frac{\partial}{\partial y} \right] \label{ThePair-XY-Introduction}
\end{eqnarray}
in the same coordinates~$(x,y)$. Here $g$ is a holomorphic function whereas $f$ is allowed to be meromorphic. The order of the
pole of~$f$, however, is such that the map $(x,y) \rightarrow x^a y^b f(x^ny^m)$ is holomorphic with order at least~$1$ at the origin.
Finally $a,b,m,n$ are all positive integers satisfying $am - bn \in \{-1,1\}$.

If in addition the functions $g$ and $f$ satisfy one of the following conditions:
\begin{itemize}
  \item $g(0) \neq 0$ or
  \item $g(0) = g'' (0) = f_0 =0$ and $g'(0) \neq 0$, where $f_0$ stands for the constant term in the Laurent
  series of $f$ around $0 \in \C$.
\end{itemize}
Then the Lie algebra generated by $X$ and $Y$ is univalent; cf. Proposition~\ref{Maximal-X-Y-TheoremA}.
In this case we have:

\vspace{0.2cm}

\noindent {\bf Theorem A}. {\sl Every pair of vector fields $X$ and $Y$ as in~(\ref{ThePair-XY-Introduction})
that satisfies one of the above conditions can be realized by complete
vector fields on a suitable (Hausdorff) open complex surface.}

\vspace{0.2cm}

The reason to focus on the above family basically stems from the fact that, in several aspects,
this family is among the most interesting choices, as indicated by the discussion in \cite{FRR}
(even if no result of \cite{FRR} will be used in the present paper, as previously pointed out).

Another natural question related to Theorem~A concerns the strategy to try and globalize (germs of) holomorphic vector fields.
A general principle to turn a pair of germs of vector fields
as above into global complete ones stems from Palais' technique \cite{Palais} and his notion of ``maximal local transformation
groups''. This method deserves additional comments, and the reader is referred to Section~2 for definitions and terminology.
First note that we can move back and forth between local transformation groups and Lie algebras of vector fields, simply recall that
a local transformation group on an open set $U$ gives rise to a representation of the corresponding Lie algebra in the
space of vector fields on $U$. Conversely, every (finite dimensional) Lie algebra of vector fields on $U$ can be integrated
to yield a local transformation group on $M$. Now consider a Lie algebra of vector fields $\mathfrak{g}$ defined on some open set $U$.
Palais considers
the problem of embedding $U$ on some manifold $M$ so that the algebra $\mathfrak{g}$ can be extended to an algebra of vector fields
in all of $M$ which integrates to give rise to an action of the (``a'') corresponding Lie group. He observed that a necessary
condition for the existence of $M$ is that the Lie algebra $\mathfrak{g}$ on $U$ verifies a certain
condition that he calls {\it univaluedness} (i.e. the local transformation group arising from $\mathfrak{g}$ must be maximal).
Then he went on to show that, when this condition is satisfied, then $M$ exists {\it in the category of non-Hausdorff manifolds}.

An issue implicitly involved in Palais' monograph \cite{Palais} is that, in general, it may be hard to detect when a given (local)
Lie algebra is univalent in his sense. The special case in which $\mathfrak{g}$ is reduced to a single holomorphic
vector field is, however, of particular importance (the condition is meaningless for a single real vector field, cf. Section~2).
Holomorphic vector fields satisfying this condition were called {\it semicomplete}\, in
\cite{JR1}, i.e. a holomorphic vector field $X$ is semicomplete if the complex one-dimensional Lie algebra generated by~$X$ alone is univalent.
Dealing with a single vector field, we shall keep the terminology ``semicomplete'' and save the word {\it univalent}\, for Lie
algebras of dimension at least~$2$: this choice provides a convenient way to discuss Lie algebras and to refer to individual vector fields
in them without risking any misunderstanding.

Concerning holomorphic vector fields, significant progress has been made in the past 20 years in terms of having criteria
to detect whether
or not a vector field is semicomplete on an open set.
It was also observed that this condition can be exploited to provide insight into the nature of singular
points of complete holomorphic vector fields (see Section~2 for additional detail). Nonetheless, relatively little is known for more general
Lie algebras, bar the case of ${\rm PSL}\, (2, \C)$ \cite{guillotIHES}; see also \cite{guillot-toulouse} for some results on abelian
Lie algebras.

Now, we can summarize some issues involved in the family of commuting vector fields considered in Theorem~A.
First, we note that $X$ is clearly semicomplete as it follows from its explicit integration (its underlying foliation is linear).
Every holomorphic vector field $Y$ commuting with $X$ should then have the form indicated in~(\ref{ThePair-XY-Introduction})
for some functions $g$ and $f$ (cf. Section~2). At this point, there is no need to pay much attention to the exact condition on $g$ and $f$,
we may just assume that $Y$ is semicomplete as well (the connection between the former conditions and the semicomplete
character of $Y$ will be detailed in Section~3).
It is useful, however, to keep in mind that the underlying foliation of
a semicomplete vector field $Y$ as before must be linearizable, though not in the {\it same coordinate as the underlying foliation
of~$X$}.
From this, there follows that either $X$ or $Y$ can easily be globalized but the question of having a (simultaneous) globalization
of the corresponding local $\C^2$-action is less simple. Actually, in principle, it is by no means clear that the (abelian)
Lie algebra formed by $X$ and $Y$ is univalent in the sense of Palais, which itself is a necessary condition for the existence
of a simultaneous globalization (cf. Proposition~\ref{Maximal-X-Y-TheoremA}).
Other ``elementary'' or ``more visible'' {\it indications}\, that the simultaneous globalization problem is
slightly more subtle are also possible. For example, the vector field $Y$ is semicomplete only on a neighborhood of
the origin and not necessarily on, say, all of $\C^2$ (even if we assume that the functions $f$ and $g$ are holomorphic everywhere).
Also, it follows from the classification of compact surfaces carrying non-trivial holomorphic vector fields accomplished by
G. Dloussky, Oeljeklaus and Toma, cf. \cite{Dloussky-1}, \cite{Dloussky-2} that this Lie algebra
cannot be globalized on a {\it compact surface}.

In general, given two commuting vector fields each of them being semicomplete, it can still be hard to show that the corresponding
Lie algebra of dimension~$2$ is univalent. In particular, to conclude that the Lie algebra generated by $X$ and $Y$
as in~(\ref{ThePair-XY-Introduction}) is univalent is not immediate and it accounts for the content of
Proposition~\ref{Maximal-X-Y-TheoremA}. In turn, to prove Proposition~\ref{Maximal-X-Y-TheoremA},
it is convenient to state a useful general result.

Let $G$ be a Lie group which is given as the semidirect product of connected Lie groups
$H$ and $K$, i.e. $G = H \ltimes K$ so that $K$ is normal in $G$ (here $H$ and $K$ are naturally identified with subgroups
of $G$). The corresponding Lie algebras can be written as $\mathfrak{g} = \mathfrak{h} \ltimes \mathfrak{k}$, see Section~3 for details.
Finally, in Theorem~B below it is not specified whether the Lie groups and manifolds are complex or real since the statement applies equally
well to both smooth and holomorphic settings. All this said, we now have:

\vspace{0.2cm}

\noindent {\bf Theorem B}. {\sl Let $G = H \ltimes K$ ($G, H, K$ connected) and consider the Lie algebra decomposition
$\mathfrak{g} = \mathfrak{h} \ltimes \mathfrak{k}$. Consider also a representation $\rho : \mathfrak{g} \rightarrow \mathfrak{X} \, (M)$
of $\mathfrak{g}$ in the space $\mathfrak{X} \, (M)$ of vector fields of some manifold $M$ and
assume the following holds:
\begin{enumerate}
  \item The Lie algebra $\rho (\mathfrak{h}) \subset \mathfrak{X} \, (M)$ integrates to an action of $H$.
  \item The Lie algebra $\rho (\mathfrak{k}) \subset \mathfrak{X} \, (M)$ integrates to a maximal local action of $K$
  (i.e. $\rho (\mathfrak{k})$ is univalent).
\end{enumerate}
Then the Lie algebra $\rho (\mathfrak{g}) \subset \mathfrak{X} \, (M)$ integrates to a maximal local action of $G$
(i.e. $\rho (\mathfrak{g})$ is univalent).}

\vspace{0.2cm}

A. Guillot should be credited with Theorem~B. In fact, Proposition~2.2 in \cite{guillotFourier} represents very much
the content of Theorem~B, all the more so as he explicitly mentions that the proof works well beyond the particular case in question.
We have independently come to a similar conclusion in order to prove
the univalent character of the Lie algebra used in our Theorem~A. Indeed, our original statement - corresponding
to Corollary~C below - was tailored to
the particular case in hand for the present paper.

\vspace{0.2cm}

\noindent {\bf Corollary C}. {\sl Let $X$ and $Y$ be a pair of commuting vector fields defined on a complex manifold $M$.
\begin{itemize}
   \item[$( \imath )$] Assume that $X$ and $Y$ are holomorphic with $X$ $\C$-complete and $Y$ semicomplete on $M$. Then the Lie
   algebra formed by $X$ and $Y$ is univalent on $M$.

   \item[$( \imath \imath )$] Assume that $X$ and $Y$ are real vector fields which, in fact, correspond to the real and to the imaginary
   parts of a holomorphic vector field $Z$ on $M$. If $X$ is $\R$-complete, then $Z$ is semicomplete on $M$.
\end{itemize}
}

\vspace{0.2cm}

We are grateful to the referee who pointed out a similar statement in the work of A. Guillot \cite{guillotFourier} and suggested us
to include here the above general formulation. Theorem~B has a few rather non-trivial - if not always immediate - applications.
A version of Theorem~B was first used in \cite{guillotmathphysics}. Here, it will be used to prove Proposition~\ref{Maximal-X-Y-TheoremA}.
It should also be pointed out that
item~$( \imath \imath )$ in Corollary~C is a variant of a well known theorem by Forstneric
\cite{F.F} stating that a $\R$-complete holomorphic vector field on $\C^n$ is automatically $\C$-complete. Forstneric's result
actually holds for complex manifolds on which every negative plurisubharmonic function is constant, but it does not hold
for arbitrary complex manifolds. This contrasts with the statement in Corollary~C, valid for every complex manifold, which shows
that the condition of $\R$-completeness always has strong consequences on the holomorphic vector field $Z$. Another curious consequence
of Theorem~B is as follows. Complete holomorphic vector fields on an open complex manifold (for example on $\C^n$) do not constitute
a Lie algebra, as the sum of two complete vector fields may fail to be complete. Yet, owing to Theorem~B, a finite collection
of complete vector fields generating a Lie algebra bound by ``affine relations'' as expressed in Theorem~B will necessarily
consist of complete vector fields.

Let us finish this introduction by briefly outlining the structure of the paper. Section~\ref{morelikebasics} contains a
review of Palais' work in \cite{Palais} along with a couple of
worked out examples.

Section~3 contains the proof of Theorem~B from which we deduce the proof of Proposition~\ref{Maximal-X-Y-TheoremA}.
Owing to Proposition~\ref{Maximal-X-Y-TheoremA}, Palais' construction
allows us to integrate the abelian Lie algebra of Theorem~A to yield a global $\C^2$-action
on a complex surface $M$ that a priori is not Hausdorff. The discussion of Palais' leaf space for vector fields
$X$ and $Y$ as in~(\ref{ThePair-XY-Introduction}) will be
carried out in Section~4, the main result being that the corresponding manifolds are always Hausdorff. Theorem~A will then
follow from what precedes.

\section{Basic issues of local nature}\label{morelikebasics}

Mostly of the material reviewed in the sequel revolves around Palais' work \cite{Palais}.
To begin, consider a holomorphic vector field $X$ defined on a possibly open complex manifold $M$.
Recall from \cite{JR1} that $X$ is said to
be semicomplete on $M$ if for every point $p \in M$ there exists an integral curve of $X$, $\phi : V_p \subset \C \rightarrow M$,
satisfying the following conditions:
\begin{itemize}

\item[(A)]  $\phi (0) =p$ and $\phi' (T) = X (\phi (T))$ for every $T \in V_p$;

\item[(B)] Whenever $\{T_i \} \subset V_p \subseteq \C$ converges to a point $\hat{T}$ in the
boundary of $V_p$, the corresponding sequence $\phi (T_i)$ leaves every compact subset of $M$.
\end{itemize}
Owing to Condition~(B), the integral curve $\phi : V_p \subset \C \rightarrow M$
is a {\it maximal solution}\, of $X$ in a sense analogue to the notion of ``maximal solutions'' commonly
used for real differential equations.

A semicomplete vector field on $M$ gives rise to a semi-global flow $\Phi$ on $M$ (see \cite{JR1}) which fits
in the setting of {\it maximal local actions}\, as considered by Palais in \cite{Palais} and also discussed in
\cite{guillotFourier}. It is therefore convenient to briefly review the main results obtained in \cite{Palais}.
In view of the objectives of this work, the discussion will mostly be conducted in the complex (holomorphic) setting
though the reader will notice that it immediately carries over the differentiable category.
Let then
$G$ denote some complex Lie group whose identity element will be denoted by $0$
(since $G$ coincides with $\C^2$ in most of our applications). General elements of $G$ will be denoted by $g$,
unless we are explicitly discussing the cases where $G$ is $\C$ or $\C^2$ where these elements will be denoted by either
$t$ or $(t,s)$ (with the standard additive notation).

\begin{defnc}
\label{LocalTransformationGroup}
A local $G$-transformation group acting on the complex manifold $M$ (or a local $G$-action on $M$)
is a holomorphic map
$$
\Phi \, : \; \, \Omega \subset G \times M \longrightarrow M \; ,
$$
where $\Omega$ is a connected open set containing $\{ 0 \} \times M$, which satisfies the following conditions:
\begin{itemize}

\item[(1)] $\Phi (0, p) =p$ for every point $p \in M$.

\item[(2)] For every pair $g_1, g_2$ in $G$ and point $p \in M$, we have
$$
\Phi (g_1 g_2, p ) = \Phi (g_1 , \Phi (g_2, p))
$$
provided that both sides are defined.
\end{itemize}
\end{defnc}

The local $G$-action is called {\it global}\, if $\Omega = G \times M$. A global action of $\C$ is called a holomorphic
flow whereas a local action of $\C$ correspond to the standard notion of {\it local flow}.

It is also clear that a local $G$-action on $M$ gives rise to a representation $\rho$ of the Lie algebra of
$G$ in the space $\mathfrak{X} \, (M)$
of holomorphic vector fields on $M$. Conversely, any representation of the Lie algebra of $G$ in $\mathfrak{X} \, (M)$ can be
integrated to yield a local $G$-action.

A local $G$-transformation group $\Phi : \; \Omega \subset G \times M \rightarrow M$ allows us to naturally
associate an open set $V_p$ of $G$ to every point $p \in M$ by letting
\begin{equation}
V_p = \{ \; g \in G \; \; : \; \; \; (g,p) \in \Omega \; \} \, . \label{definitionV_p}
\end{equation}
Similarly, the orbit $\mathcal{O}_p \subset M$ of $p$ under $G$ is defined by
$$
\mathcal{O}_p = \{ \; \Phi (g,p) \; \; : \; \; \; (g,p) \in \Omega \; \} \, .
$$
In other words, $\mathcal{O}_p$ is the image of the above defined set $V_p \subset G$ by $\Phi$.

The reader will note that the partition of $M$ in orbits of $G$ naturally endows $M$ with a structure of singular foliation.

The fundamental notion discussed by Palais in \cite{Palais} is that of {\it maximal local action}\, and it can be formulated as follows.

\begin{defnc}
\label{MaximalTransformationGroup}
Consider a local $G$-transformation group $\Phi : \; \Omega \subset G \times M \rightarrow M$ and for every
point $p \in M$, let $V_p \subset G$ be as in~(\ref{definitionV_p}). The local
$G$-transformation group $\Phi : \; \Omega \subset G \times M \rightarrow M$ is said to be maximal
if it satisfies the following additional condition: for every $p\in M$ and for every
sequence $\{ g_i \} \subset V_p$ converging towards a point $\widehat{g}$ lying in the boundary of $V_p \subset G$,
the sequence $\Phi (g_i, p)$ leaves every compact subset of $M$.
\end{defnc}

Alternative formulations of Definition~\ref{MaximalTransformationGroup} can be obtained from Theorem~6 in
\cite{Palais} (pages 66-67). Note also that the definition reduces to that of semi-global flow in the case $G = \C$,
cf. \cite{JR1}.


Building on what precedes, let us now discuss the globalization problem according to \cite{Palais}.
First, to abridge notation, we assume from now on that every representation of a (finite dimensional) Lie algebra
in the Lie algebra $\mathfrak{X} \, (M)$ is faithful. In practice, we will directly deal with complex
Lie sub-algebras $\mathfrak{g}$ of $\mathfrak{X} \, (M)$. Recall that up to identifying $\mathfrak{g}$ with
the Lie algebra of a Lie group $G$, the vector fields in $\mathfrak{g}$ can always be integrated to
yield a local action of $G$ (local $G$-action) on $M$. In the special case where $M$ is compact, this procedure
actually yields a (global) action of $G$ (or a faithful action of a quotient of $G$). For this reason,
we can restrict our attention to the case where $M$ is an open manifold.

\begin{defnc}
\label{MaximalLiealgebras}
The Lie algebra $\mathfrak{g} \subset \mathfrak{X} \, (M)$ is said to be univalent if it can be integrated
to give rise to a maximal local transformation group $G$ on $M$.
\end{defnc}

As previously mentioned, in the case of a single (holomorphic) vector field $X$, we will simply say that $X$
is semicomplete, rather than saying that the Lie algebra of dimension~$1$ generated by $X$ is univalent. In this
way, the word ``univalent'' will be saved for Lie algebras of dimension at least~$2$.

The reader will note that univalent Lie algebras are stable by restriction: if $\mathfrak{g}$ is univalent on $M$ and
$U$ is an open set of $M$, then the restriction of $\mathfrak{g}$ to $U$ is naturally univalent as well (\cite{Palais}, \cite{JR1}).
In particular, we can talk about {\it germs} of univalent Lie algebras.

To make the discussion accurate, we begin with a couple of definitions slightly adapted from \cite{Palais}.

\begin{defnc}
\label{Globalization_definition}
Let $M$ be an open complex manifold and let $\mathfrak{g}$ be a finite dimensional complex Lie
(sub-) algebra of $\mathfrak{X} \, (M)$.
A globalization of the pair $(\mathfrak{g}, M)$
is a pair $(\overline{\mathfrak{g}}, \overline{M})$ satisfying all of the following conditions:
\begin{itemize}
\item[(1)] $\overline{M}$ is a complex manifold and $\overline{\mathfrak{g}}$ is a Lie (sub-) algebra of
$\mathfrak{X} \, (\overline{M})$ (which is isomorphic to $\mathfrak{g}$ as abstract Lie algebra).

\item[(2)] $\overline{\mathfrak{g}}$ gives rise to a (global) action $\Phi : G \times \overline{M}
\rightarrow \overline{M}$ of a certain Lie group $G$.

\item[(3)] There is a holomorphic diffeomorphism $\psi : M \rightarrow U$, where $U$ is some open set
of $\overline{M}$.

\item[(4)] The differential of $\psi$ sends $\mathfrak{g}$ to the restriction of $\overline{\mathfrak{g}}$
to $U$. In other words, pull-back by $\psi$ yields an one-to-one correspondence between the restriction to $U$ of vector fields
in $\overline{\mathfrak{g}}$ and vector fields defined on $M$ and belonging to $\mathfrak{g}$.

\item[(5)] For every point $p \in \overline{M}$, there exists $g \in G$ such that
$\Phi (g, p)$ lies in $U$.

If the pair $(\overline{\mathfrak{g}}, \overline{M})$ satisfies conditions~1--4 above but falls short from satisfying condition~5,
then we will say that $(\overline{\mathfrak{g}}, \overline{M})$ is a realization of the pair $(\mathfrak{g}, M)$.
\end{itemize}
\end{defnc}

\begin{obs}
\label{VariousComments-onDefinition}
{\rm Definition~\ref{Globalization_definition} needs a few additional comments. First it is clear that any globalization
of the pair $(\mathfrak{g}, M)$ is also a realization of $(\mathfrak{g}, M)$. This observation also admits a partial converse, namely
every realization of $(\mathfrak{g}, M)$ contains a globalization of $(\mathfrak{g}, M)$. Indeed consider a realization
$(\overline{\mathfrak{g}}, \overline{M})$ of $(\mathfrak{g}, M)$. Up to identifying (by $\psi$)
$M$ with an open set $U$ of $\overline{M}$, a globalization
for $(\mathfrak{g}, M)$ can be obtained by simply taking the saturated of $U$ under the action of $G$ on $\overline{M}$.
The interest of including condition~(5) in the notion of ``globalization'' will, however, become clear in the
course of our discussion.

Still concerning Definition~\ref{Globalization_definition} another issue worth mentioning has to do with identifications of
$\mathfrak{g}$ and $\overline{\mathfrak{g}}$. As mentioned, they are isomorphic as Lie algebras, indeed they are both isomorphic
to the Lie algebra of a certain Lie group $G$. Yet, Definition~\ref{Globalization_definition} assumes no identification
between the two algebras in question other than the one established by $\psi$. In this respect, the reader will note that
Definition~\ref{Globalization_definition} does not mention any a priori local $G$-action on $M$.

Naturally, if one such identification is chosen from the beginning, then we may also ask about the existence of an {\it equivariant}\,
globalization. Additional information in this direction can be found in Remark~\ref{anti-equivariant}. Note, however, that this type
of issue does not pertain the statement of our Theorem~A.}
\end{obs}

As previously mentioned, the univalent character of $\mathfrak{g}$ on $M$ is a necessary condition for the existence
of a globalization (and hence of a realization) of the pair $(\mathfrak{g}, M)$.

Note that conditions~(1) through~(5) in Definition~\ref{Globalization_definition} make sense even in
the context of non-Hausdorff manifolds used in \cite{Palais}. Thus, it is convenient to introduce a non-Hausdorff analogue
of Definition~\ref{Globalization_definition}.

\begin{defnc}
\label{nHausdorff_Globalization_definition}
Consider a pair $(\mathfrak{g}, M)$ as in Definition~\ref{Globalization_definition}. A
nH-globalization (resp. $nH$-realization) of $(\mathfrak{g}, M)$ is a pair $(\overline{\mathfrak{g}}, \overline{M})$,
where $\overline{M}$ is a possibly non-Hausdorff complex manifold, satisfying conditions~(1)--(5)
(resp. conditions~(1)--(4)) of Definition~\ref{Globalization_definition}.
\end{defnc}

As a matter of fact, one of the main applications of the study of germs of semicomplete vector fields
is precisely to provide insight
in the structure of singular points of complete vector fields. Namely, if a {\it germ of} vector field is not semicomplete
then it cannot appear as the singular point of a complete vector field on any manifold. Since a similar study and application
can be envisaged for more general Lie algebras,
it is interesting to adapt the content of Definition~\ref{Globalization_definition} to the case of germs.

Consider then the germ of a Lie algebra $\mathfrak{g}$ at a point $p$ of a manifold $M$. A pair $(\overline{\mathfrak{g}}, \overline{M})$
will be called a {\it realization}\, of the germ of $\mathfrak{g}$ at $p$ if the following holds:
\begin{itemize}
  \item There is an open set $U \subset M$ equipped with a Lie algebra of vector fields denoted by $\mathfrak{g}_U$
  which represents the germ of $\mathfrak{g}$ at $p$.

  \item $(\overline{\mathfrak{g}}, \overline{M})$ is a realization of the pair $(\mathfrak{g}_U, U)$.
\end{itemize}

The notion of globalization of a germ of vector fields is slightly more subtle. Consider again a germ $\mathfrak{g}$
of Lie algebra at a point $p$ in a manifold $M$.

\begin{defnc}
\label{Globalization_definition_GERMS}
The pair $(\overline{\mathfrak{g}}, \overline{M})$ is said to be the {\it globalization}\, of $\mathfrak{g}$ at $p$ if there
are a pair $(\mathfrak{g}_U, U)$ and a decreasing sequence of
open set $U_k \subset U$, $k \in \N$, such that:
\begin{itemize}
  \item The pair $(\mathfrak{g}_U, U)$ represents the germ of $\mathfrak{g}$ at $p$.
  \item $\bigcap_{k \in \N} U_k = \{ p \}$.
  \item For every $k \in \N$, the pair $(\overline{\mathfrak{g}}, \overline{M})$ is a globalization of the pair
  $(\mathfrak{g}_{U \vert_{U_k}}, U)$ where $\mathfrak{g}_{U \vert_{U_k}}$ stands for the restriction of
  $\mathfrak{g}_U$ to $U_k$.
\end{itemize}
\end{defnc}

\begin{obs}
{\rm Unlike the case in which a Lie algebra of vector fields on a given manifold $M$ is considered, in the
context of germs, it is not true in general that a realization of a germ $\mathfrak{g}$ necessarily contains a
{\it globalization} of $\mathfrak{g}$, cf. Remark~\ref{Palais_germification}.}
\end{obs}

We can now summarize Palais' construction of a (possibly non-Hausdorff)
globalization for every pair $(\mathfrak{g}, M)$ where $\mathfrak{g}$ is a Lie algebra of vector fields on $M$ having
finite dimension.
It will turn out that Palais globalization $(\overline{\mathfrak{g}}, \overline{M})$ of
$(\mathfrak{g}, M)$ arises as a leaf space for a suitable foliation. Thus, whereas this leaf space is endowed with
charts giving it a structure of ``complex manifold'', the topology of the underlying topological space may fail to
be Hausdorff.

Let then $(\mathfrak{g}, M)$ be as above, i.e. $\mathfrak{g}$ is a finite dimensional Lie algebra of (holomorphic)
vector fields on a complex manifold $M$. We assume in the sequel that $\mathfrak{g}$ is univalent on $M$.
As usual, the space of (holomorphic) vector fields on $M$ will be denoted by
$\mathfrak{X} \, (M)$.

The Lie algebra
$\mathfrak{g}$ yields by integration a local transformation group $G$ on $M$ whose (maximal) local action will be denoted
by $\Phi : \Omega \subset G \times M \rightarrow M$. In turn, $\mathfrak{g}$ can also be identified with the
Lie algebra of $G$ and hence with left-invariant vector fields on $G$. Thus, to every vector field
$X \in \mathfrak{g} \subset \mathfrak{X} \, (M)$, there corresponds a left-invariant vector field
$X_G \in \mathfrak{X} \, (G)$ obtained through the natural identification of the Lie algebra of $G$ with the space
$\mathfrak{X} \, (G)$ of left-invariant vector fields on $G$.

\begin{obs}
{\rm In connection with the second paragraph of Remark~\ref{VariousComments-onDefinition}, the reader will note that
we have just fixed an identification of $\mathfrak{g}$ with the Lie algebra of $G$. The local action of $G$ on $M$
was also mentioned so that we will be able to make sense of equivariance-related properties of the globalization to be
constructed; see Remark~\ref{anti-equivariant}.}
\end{obs}

Next, consider the manifold $N = G \times M$ along with the embedding $\psi : M \rightarrow N$ given by
the identification $M \simeq (0, M) \subset N$ (where $0$ stands for the neutral element of $G$).
In the sequel, we can think either in terms of Lie algebras or of transformation
groups. In terms of Lie algebras, there are two natural representations of $\mathfrak{g}$ in $\mathfrak{X} \, (N)$
(where $\mathfrak{X} \, (N)$ stands for the space of holomorphic vector fields on $N$).
Namely, we define $\theta_1 : \mathfrak{g} \rightarrow \mathfrak{X} \, (N)$ by letting
$$
\theta_1 \, (X) \, (g, p) = (X_G (g), 0) \in T_{(g,p)} N \simeq T_gG \times T_p M \, .
$$
Naturally the Lie (sub-) algebra $\theta_1 (\mathfrak{g}) \subset \mathfrak{X} \, (N)$ integrates to an action
of $G$ on $N$, namely the action of $G$ on itself. More precisely, the action in question is given by the map
$\Psi : G \times N = G \times G \times M \rightarrow N =G \times M$ defined by
letting $\Psi (g_1,g_2,p) = (g_1g_2, p) \in N = G \times M$.

The second natural representation arises from the map $\theta_2 : \mathfrak{g} \rightarrow \mathfrak{X} \, (N)$
assigning to a vector field $X \in \mathfrak{g}$ the vector field $\theta_2 \, (X) \in \mathfrak{X} \, (N)$
defined by
$$
\theta_2 \, (X) \, (g, p) = (X_G (g), X (p)) \in T_{(g,p)} N \simeq T_g G \times T_p M \, .
$$
Clearly, both $\theta_1$ and $\theta_2$ are isomorphisms of Lie algebras from $\mathfrak{g}$ to its respective images
$\theta_1 \, (\mathfrak{g})$ and $\theta_2 \, (\mathfrak{g})$.

Now, the algebra $\theta_2 \, (\mathfrak{g}) \subset
\mathfrak{X} \, (N)$ can be integrated to yield a local action of $\widetilde{\Phi}$ of
$G$ on $N$ whose orbits define the leaves of a foliation $\fol_N$ on $N$.
Furthermore, the local actions $\Phi$ and $\widetilde{\Phi}$ of $G$ on $M$ and of $G$
on $N$ are equivariant with respect to the natural projection $N = G \times M \rightarrow M$.
It also follows that $\theta_2 \, (\mathfrak{g})$ is univalent on $N$ since $\mathfrak{g}$ is univalent on $M$.

Next consider the leaf space $\overline{M}$ of $\fol_N$ equipped with the natural quotient topology
inherited from $N$. Whereas this topology fails in general to be Hausdorff, the leaf space $\overline{M}$ itself
possesses natural complex coordinates giving it the structure of a non-Hausdorff complex manifold. We also denote
by ${\rm Proj}$ the canonical projection of $N$ to $\overline{M}$.

Now, by construction, $\Psi$ preserves the foliation $\fol_N$ and hence
induces a (global) action of $G$ on
$\overline{M}$. Similarly, $\theta_1 \, (\mathfrak{g})$ projects onto a Lie algebra
$\overline{\mathfrak{g}}$ of vector fields on $\overline{M}$. Naturally
$\overline{\mathfrak{g}}$ is again isomorphic to $\mathfrak{g}$ and, indeed, the isomorphism is fixed from the beginning:
we denote by $\eta$ this isomorphism. Finally it is clear that
$\overline{\mathfrak{g}}$ integrates to the above mentioned action of $G$ on $\overline{M}$.

It is convenient to explicitly work out the above defined action of $G$ on $\overline{M}$. This will enable
us to compare the isomorphism $\eta$
between $\mathfrak{g}$ and $\overline{\mathfrak{g}}$ with the isomorphism arising from the embedding $\psi$ given by
$M \simeq (0,M) \subset N$.
In doing this, it will be shown not only
that $(\overline{\mathfrak{g}}, \overline{M})$ is a realization of the pair $(\mathfrak{g}, M)$ in the sense of Definition~\ref{Globalization_definition}
but also that $\psi$ is not equivariant with respect to the above mentioned actions of $G$.

For this, note first that $\fol_N$ is transverse to $M \simeq (M,0)$. In addition,
each leaf of $\fol_N$ intersects $M \simeq (0,M)$ at a single point (at most). The second assertion follows
from Theorem~6 in \cite{Palais} (pages 66-67) since $\theta_2 \, (\mathfrak{g})$ is known to be univalent on~$N$.
Therefore, the embedding $\psi$ of $M$ in $N$ ($M \simeq (0,M)$) naturally induces an embedding (diffeomorphism) from
$M$ to the leaf space $\overline{M}$ of $\fol_N$. In slightly more accurate terms, the composition ${\rm Proj} \circ \psi$
provides an embedding from $M$ to $\overline{M}$.

Finally, fix a point $p = (0,p) \in M \simeq (0,M) \subset N$ and denote by $\overline{p}$ the corresponding point in
$\overline{M}$. Given a vector field $X \in \mathfrak{g}$ and let us consider the action of its exponential
$g_t = \exp (tX) \subset G$ in $\overline{M}$, for $t$ small. The local flow of $\theta_1 (X)$ moves the point $(0,p)$
to the point $(g_t, p)$, i.e. $\Psi (g_t, (0,p)) = (g_t ,p)$. On the other hand,
the leaf $L_{(g_t,p)}$ of  $\fol_N$ through $(g_t, p)$ intersects $(0,M)$ at a point~$(0,q)$
associated with a point $\overline{q} \in \overline{M}$. Hence the action of $g_t$ on $\overline{M}$ takes $\overline{p}$
to $\overline{q}$. However, by construction, $(0,q)$ is such that $\widetilde{\Phi} (g_t, (0, q)) = (g_t ,p)$.
The equivariant nature of $\Phi$ and $\widetilde{\Phi}$ with respect to the projection $N = G \times M \rightarrow M$
then implies that $\Phi (g_t ,q) =p$. Summarizing, we have $\Phi (g_t ,q) =p$ whereas the action induced by $\Psi$ on
$\overline{M}$ is such that the element $g_t$ takes $\overline{p}$ to $\overline{q}$. Therefore we conclude that, with
respect to the diffeomorphism associated with the embedding ($\psi$) $M \simeq (0,M) \subset N$ takes the vector field $X$
to the vector field $-X$, with respect to the (previously fixed) isomorphism $\eta$ between $\mathfrak{g}$ and
$\overline{\mathfrak{g}}$. This completes Palais' construction of the realization of $(\mathfrak{g}, M)$.

\begin{obs}
\label{anti-equivariant}
{\rm Recalling that ${\rm Proj} \circ \psi$ provides an embedding of $M$ in $\overline{M}$, let
${\rm Proj} \circ \psi (M) = \overline{U} \subset \overline{M}$. The above
construction shows, in particular, that pulling-back by ${\rm Proj} \circ \psi$ establishes an one-to-one correspondence between vector fields
in $\mathfrak{g} \subset \mathfrak{X} \, (M)$ and the restriction to $\overline{U} \subset \overline{M}$ of the Lie algebra
$\overline{\mathfrak{g}}$. Since $\overline{\mathfrak{g}}$ actually gives rise to an action of $G$ on $\overline{M}$,
we have obtained a realization of the pair $(\mathfrak{g}, M)$ in the sense of Definition~\ref{Globalization_definition}.
Furthermore, it can easily be checked that the pair $(\overline{\mathfrak{g}}, \overline{M})$ is, in fact, a
{\it globalization}\, of $(\mathfrak{g}, M)$.

On a different notice, having fixed from the beginning an identification of $\mathfrak{g}$ with the Lie algebra of a group $G$ and
also considered the corresponding local action of $G$ on $M$, there is a sense to ask whether the embedding
${\rm Proj} \circ \psi$ is equivariant with respect to the action of $G$ on $M$ and the action of $G$ on $\overline{M}$
(these definition of these actions being directly related to the isomorphism $\eta$ between $\mathfrak{g}$ and
$\overline{\mathfrak{g}}$). The answer
is no, since ${\rm Proj} \circ \psi$ takes $X$ to $-X \in \overline{\mathfrak{g}}$. In this sense, Palais' construction
actually yields an anti-equivariant globalization of $(\mathfrak{g}, M)$.}
\end{obs}

Palais' construction possesses at least one additional property that is interesting. Namely, the globalization
is {\it universal}\, in a suitable sense. We shall not review this property here since it will not be used in the
remainder of the paper; the reader is referred to Theorem~9 (page 71)
of \cite{Palais}.

\begin{obs}
\label{Palais_germification}
{\rm Note that Palais' construction does not provide a globalization for germs of Lie algebras. Indeed,
consider a Lie algebra $\mathfrak{g}$ defined (and univalent) on an open set $U_1$ and denote by
$(\overline{\mathfrak{g}}, \overline{U}_1)$ the corresponding Palais globalization. Next, let
$(\overline{\mathfrak{g}}, \overline{U}_2)$ be the Palais globalization of the restriction of $\mathfrak{g}$
to another open set $U_2$ contained in $U_1$. Note that the above construction provides no natural embedding of
$\overline{U}_2$ in $\overline{U}_1$ so that the ``stability part'' in the definition of the globalization
of germs (Definition~\ref{Globalization_definition_GERMS}) is not fulfilled in general; cf. Proposition~\ref{prop_example}
below.}
\end{obs}

Let us close this section with a couple of worked out examples of Palais' construction.
Probably the simplest example involving holomorphic vector fields is provided by
the one-dimensional vector field $X = x^2 \partial /\partial x$ which is
regarded as defined on a neighborhood $U$ of $0 \in \C$. Clearly $X$ is also semicomplete
on $U$ since it is so on all of $\C$. Whereas it is an easy consequence of Riemann's theorem that
$\C P (1)$ is the only Riemann
surface on which $X$ can be extended to a complete holomorphic vector field, Palais' construction leads directly
to the same conclusion.

To check the preceding assertion, consider the vector field
\[
\overline{X} = \frac{\partial}{\partial t} + X = \frac{\partial}{\partial t} + x^2 \frac{\partial}{\partial x}
\]
defined on $\C \times U$, where $U$ is the previously fixed neighborhood of $0 \in \C$. Denoting by $\fol_{\overline{X}}$
the foliation associated with $\overline{X}$, the above assertion amounts to checking that the leaf space of
$\fol_{\overline{X}}$ can be identified with $\C P(1)$. For this, however, it suffices to observe that this leaf space
is realized by the map $\pi : \C \times U \rightarrow \C P(1)$ given by $\pi (t,x) = [x:tx+1]$.

A slightly more subtle example which illustrates the content of Remark~\ref{Palais_germification} is formulated
as Proposition~\ref{prop_example} below.

\begin{prop}\label{prop_example}
Consider the vector field $X$ defined on $\R^2$ by
\[
X = 4y \frac{\partial}{\partial x} - x \frac{\partial}{\partial y} \, .
\]
Then, for every open neighborhood $U$ of the origin of $\R^2$ there exist open sets $U_1$ and $U_2$, with $0 \in
U_1 \subseteq U$ and $0 \in U_2 \subseteq U$ such that Palais globalization of $X$ for the pair $(X,U_1)$ (resp. $(X,U_2)$)
is Hausdorff (resp. non-Hausdorff). Here $(X,U_1)$ (resp. $(X,U_2)$) stands for the restriction of $X$ to $U_1$ (resp. $U_2$).
\end{prop}

\begin{proof}
First note that $X$ is semicomplete on every open set of $\R^2$ since this is the case for every vector field giving rise
a local action of $\R$. Next, the function
$H(x,y) = x^2 + 4y^2$ is a first integral for $X$ so that the leaves of the foliation
$\fol_X$ associated with $X$ are ellipses given by the equation
\[
x^2 + 4y^2 = k \, ,
\]
for $k \in \R_0^+$.

Fix an open neighborhood $U$ of the origin of $\R^2$ and let $\varepsilon$ be a small positive real constant such
that $B(0,\varepsilon) \subseteq U$. Let
\[
U_1 = \{(x,y) \in \R^2 : \, x^2 + 4y^2 < \varepsilon^2\}
\]
and
\[
U_2 = \{(x,y) \in \R^2 : \, x^2 + y^2 < \varepsilon^2\} = B(0,\varepsilon) \, .
\]
Let then $N_1 = \R \times U_1$ and $N_2  = \R \times U_2$ and consider the vector field $\overline{X} = \frac{\partial}{\partial y} + X$
on $N_1$ and on $N_2$.

Since $U_1$ is totally invariant by the flow of $X$, the action of $\R$ arising from the (local) flow of $X$ is already global
in $U_1$. It then becomes clear that the leaf space associated with $\overline{X}$ on $\R \times U_1$ in Palais' construction
is $U_1$ itself: in particular it is Hausdorff.

The same does not happen for the open set $U_2$. In fact, the set of points of $U_2$ satisfying the equation $x^2 + 4y^2 =
\varepsilon^2$ has two distinct connected components and they should be viewed as two distinct leaves of $\fol_X$ on $U_2$.
In turn, for every $0 < \delta < \varepsilon$ arbitrarily small, the set of points of $U_2$ satisfying $x^2 + 4y^2 = \varepsilon^2
- \delta$ is a leaf of $\fol$ that is constituted by a unique connected component.

Let now $L_+$ and $L_-$ stand for the leaves of $\overline{X}$ through the points $(0,0,\varepsilon/2)$ and $(\pi/2,0,
-\varepsilon/2)$. The argument presented in the previous paragraph implies that $L_+$ and $L_-$ are distinct leaves of
the foliation $\fol_{\overline{X}}$ induced on $N_2$ by $\overline{X}$. They cannot, however, be separated by two open sets
in the space of leaves of $\fol_{\overline{X}}$. Indeed, for every fixed neighborhood of the point corresponding to $L_+$
in the mentioned leaf space, there exists $\delta_+ > 0$ such that the neighborhood in question contains the points
corresponding to the leaves of $\fol_{\overline{X}}$ through points of the form $(0,0,\varepsilon/2 - \delta)$ for
every $\delta \in (0, \delta_+)$.
Analogously, every neighborhood of the point in the leaf space of $\fol_{\overline{X}}$ corresponding to $L_-$
contains the points corresponding to the leaves of $\fol_{\overline{X}}$ through $(\pi/4,0,-\varepsilon/2 + \delta)$,
provided that $\delta$ is small enough.
However, the leaf of $\fol_{\overline{X}}$ through
the point $(0,0,\varepsilon/2 - \delta)$ coincides with the leaf through the point $(\pi/4,0,-\varepsilon/2 + \delta)$.
The proposition follows.
\end{proof}

\section{Theorem~B and some applications}\label{Section_examples}

This section contains the proof of Theorem~B as well as Proposition~\ref{Maximal-X-Y-TheoremA}
concerning the univalent character of the Lie algebra generated by the vector fields $X$ and $Y$ as in
Theorem~A.

As usual, we can think either in terms of groups or in terms of Lie algebras. For our purposes, however, it seems
more convenient to begin Lie algebras.

Consider a (finite dimensional) Lie algebra of vector fields $\mathfrak{g} \subset \mathfrak{X} \, (M)$, where
$\mathfrak{X} \, (M)$ stands for the space of vector fields on a manifold $M$. Note that we need not distinguish between
the smooth or the holomorphic settings in what follows (the corresponding Lie algebras being accordingly
considered over $\R$ or over $\C$). Assume we have two (sub-) Lie algebras $\mathfrak{h}$ and $\mathfrak{k}$ of
$\mathfrak{g}$ satisfying the following conditions:
\begin{enumerate}
  \item As vector space, we have $\mathfrak{g} = \mathfrak{h} \oplus \mathfrak{k}$. In other words, every vector field
  in $\mathfrak{g}$ admits a unique decomposition as the sum of a vector field lying in $\mathfrak{h}$ and a vector field
  lying in $\mathfrak{k}$.

  \item The lie algebra $\mathfrak{k}$ is an ideal of $\mathfrak{g}$. In other words, for every $X \in \mathfrak{g}$ and
  every $Y \in \mathfrak{k}$, the commutator $[X,Y]$ belongs to $\mathfrak{k}$.
\end{enumerate}
The existence of the Lie (sub-) algebras $\mathfrak{h}$ and $\mathfrak{k}$, gives $\mathfrak{g}$ the structure of a
semidirect product as follows. For $Z \in \mathfrak{h}$, note that the adjoint representation
${\rm ad}\, (X)$ stabilizes $\mathfrak{k}$ since the latter is an ideal of $\mathfrak{g}$ and $\mathfrak{h} \subset
\mathfrak{g}$. Thus, the adjoint representation provides a homomorphism $\sigma : \mathfrak{h} \rightarrow {\rm Der}\, (\mathfrak{k})$
where $\sigma (Z)$ is the derivation of $\mathfrak{k}$ given assigning the commutator $[Z,Y]$ to the vector field $Y \in
\mathfrak{k}$. In other words, for $Z \in \mathfrak{h}$ and $Y \in \mathfrak{k}$, we have
$[Z,Y] = [\sigma (Z)] (Y)$. The triplet $(\mathfrak{h}, \mathfrak{k}, \sigma)$ then yields a unique structure of Lie algebra
on the direct sum $\mathfrak{h} \oplus \mathfrak{k}$ which is compatible with the Lie algebra structures in $\mathfrak{h}$
and in $\mathfrak{k}$. This structure is the so-called semidirect product of $\mathfrak{h}$ and $\mathfrak{k}$ (with respect
to the homomorphism $\sigma : \mathfrak{h} \rightarrow {\rm Der}\, (\mathfrak{k})$). In the present case, it coincides
with the initial Lie algebra structure in $\mathfrak{g} = \mathfrak{h} \oplus \mathfrak{k}$. The preceding is then
summarized by writing $\mathfrak{g} = \mathfrak{h} \ltimes \mathfrak{k}$, i.e. $\mathfrak{g}$ is the semidirect
product of its sub-algebras $\mathfrak{h}$ and $\mathfrak{k}$.

Let us now assume that $\mathfrak{g} = \mathfrak{h} \ltimes \mathfrak{k}$ as above satisfies the conditions in Theorem~B, namely:
\begin{itemize}
  \item $\mathfrak{h} \subset \mathfrak{X} \, (M)$ integrates to an action of a Lie group $H$ on $M$.
  \item $\mathfrak{k} \subset \mathfrak{X} \, (M)$ integrates to a maximal local action of a Lie group $K$ on $M$.
  (i.e. $\mathfrak{k}$ is univalent).
\end{itemize}
In turn, the Lie algebra $\mathfrak{g}$ integrates to a local action of a group $G$.

Some comments about the corresponding Lie groups $H$, $K$, and $G$ are needed. The action of $H$ on $M$ will be denoted by
$\Psi$ ($\Psi : H \times M \rightarrow M$) while $\Phi$
($\Phi : \Omega \subset K \times M \rightarrow M$) will stand for the maximal local action of $K$ on $M$.

The group $H$ has also a natural action $\tau : H \times K \rightarrow K$ on the group $K$ which is given by
$\tau(h,k) = hkh^{-1}$. Indeed, every element $h$ of $H$ can be identified with a diffeomorphism of $M$ which coincides
with the exponential (i.e. the time-one map) of a certain vector field $Z_h \in \mathfrak{h}$. Analogously,
every element $k \in K$ can be identified with the time-one map arising from a vector field $Y_k \in \mathfrak{k}$.
This time-one map is not necessarily globally defined as transformation of $M$ but it constitutes a diffeomorphism from its domain
to its image.
Since $[Z_h , Y_k]$ belongs to the Lie algebra $\mathfrak{k}$, it follows from the so-called Hadamard lemma
(\cite{serre}) that $h$ preserves $\mathfrak{k}$: in slightly more concrete terms, the Lie algebra $\mathfrak{k}$ is
stable under pull-backs by $h$ viewed as the time-one map of $Z_h$ (where $h \in H$ is identified with the corresponding
transformation of $M$ which is a diffeomorphism from $M$ to $M$).

Up to identifying elements of $H$ and $K$ with the corresponding transformations of $M$, the preceding also implies that
the conjugation $h \circ k \circ h^{-1}$ of a transformation $k \in K$ by a transformation (global diffeomorphism of $M$)
$h \in H$ necessarily belongs to $K$ (i.e. it is a transformation of $M$ previously identified with some element of $K$).
This shows that $\tau : H \times K \rightarrow K$ is well defined.

On a similar note, every element (transformation of $M$)
$g \in G$ is the time-one map induced by a vector field of the form $X+Y = Y+X$ with
$X \in \mathfrak{h}$ and $Y \in \mathfrak{k}$. Therefore the standard formulas revolving around Campbell-Hausdorff
(\cite{serre}) show
that every element $g$ can be written as a composition $h \circ k$ for uniquely defined $h \in H$ and $k \in K$. This gives rise
to the isomorphism $G = H \ltimes K$ (cf. below). Similarly, the same formulas also show that $g$ can be decomposed
as $k^{\ast} \circ h$ for uniquely defined $k^{\ast} \in K$ and $h \in H$ (note that the element $h$ in $H$ remains the same whether
$g$ is decomposed as $h \circ k$ or as $k^{\ast} \circ h$).
This second form of representing elements in $G$ as a pair of elements
in $H$ and in $K$ leads to the isomorphism $G = K \rtimes H$ (which, in turn, provides an analogue decomposition of
$\mathfrak{g}$ as $\mathfrak{g} = \mathfrak{k} \rtimes \mathfrak{h}$).
In what follows, we shall write $g = h \circ k$ or
$g = hk$ (resp. $g = k^{\ast} \circ h$ or $g = k^{\ast} h$) accordingly to whether or not we want to emphasize that
$g$ should be identified with the corresponding transformation of $M$.

It is also worth pointing out that the identity
$g = h \circ k = k^{\ast} \circ h$ ensures, in particular, that the domain of definition
of $h \circ k$ coincides with the domain of definition of $k^{\ast} \circ h$ (and both domains coincide with the domain of definition of $g$).
The domain of definition of $h \circ k$ is nothing but the domain of definition of $k$ whereas the domain of definition
of $k^{\ast} \circ h$ is $h^{-1} (V_{k^{\ast}})$ where $V_{k^{\ast}} \subset M$ is the domain of definition of $k^{\ast}$.

We can now recover the Lie group $G$ obtained by integrating the Lie algebra $\mathfrak{g}$ as the semidirect product
of $H$ and $K$. The underlying manifold
of $G$ is simply the product $G = H \times K \simeq K \times H$ of the underlying manifolds of $H$ and $K$.
The product of $G$ as group is defined as
\begin{equation}
(h_1 , k_1) (h_2, k_2) = (h_1 h_2 , \tau (h_2^{-1}, k_1) k_2) = (h_1 h_2 , (h_2^{-1} k_1 h_2)k_1) \; ,
\label{ThesemidirectProduct-A}
\end{equation}
for $G = H \ltimes K$. For the decomposition $G =K \rtimes H$, the product becomes
\begin{equation}
(k^{\ast}_1,h_1) (k^{\ast}_2, h_2) = (k^{\ast}_1 (h_1 k^{\ast}_2h_1^{-1}), h_1 h_2) \, . \label{ThesemidirectProduct-B}
\end{equation}

\begin{proof}[Proof of Theorem~B]
The argument is basically already encoded in the above discussion. From a technical point of view, the
proof of Theorem~B amounts to constructing a maximal local action for $G$ on $M$. For this, we proceed
as follows.

Recall that $\Phi : \Omega \subset K \times M \rightarrow M$ is the maximal local action of $K$ on $M$.
For each $q \in M$, let $\Omega_q \subset K$ be the set defined by
$$
\Omega_q = \{ \, k \in K \; ; \; \; (k,q) \in \Omega \, \} \, .
$$
Now, we define the set $\mathcal{U} \subset G \times M = K \times H \times M$ ($G = K \rtimes H$) by
$$
\mathcal{U} = \{ \, (k^{\ast},h,q) \in K \times H \times M \; ; \; \; k^{\ast} \in \Omega_{h(q)} \, \} \, ,
$$
where the fact that $h$ is a globally defined diffeomorphism of $M$ was implicitly used. If
instead of the decomposition $g = k^{\ast} \circ h$, the decomposition $g = h \circ k$ is used, then there follows
that the domain of definition of $g = h \circ k$ coincides with the domain of definition of $k$. For reference, let us
state this observation as a claim:

\vspace{0.1cm}

\noindent {\it Claim}. A point $(\overline{k}^{\ast}, \overline{h}, p)$ belongs to
the boundary $\partial \mathcal{U}$ of $\mathcal{U}$ if and only if $\overline{k}^{\ast}$ belongs to
the boundary of $\Omega_{\overline{h}(q)} \subset K$. Similarly, for the decomposition $G = H \ltimes K$, the
point $(\overline{h}, \overline{k} ,p)$ lies in $\partial \mathcal{U}$ if and only if
$\overline{k}$ lies in the boundary $\partial \Omega_p$ of $\Omega_p$.\qed

On the other hand, there is a naturally defined map $\Upsilon : \mathcal{U} \rightarrow M$ whose
value on $(g,q) = (k^{\ast},h,q)$ is
$$
\Upsilon (g,q) = \Upsilon (k^{\ast},h,q) = k^{\ast} \circ h(q) = \Phi (k^{\ast}, \Psi (h,q)) \, .
$$
The proof of Theorem~B is reduced to checking that $\Upsilon$ is a maximal local action arising from $G = K \rtimes H$.
We first note that $\Upsilon$ is a local action of $G$ on $M$. Indeed, $\Upsilon$ clearly satisfies the first
condition in Definition~\ref{LocalTransformationGroup}. To see that $\Upsilon$ satisfies the second condition as well
just note that
\begin{eqnarray*}
\Upsilon (k^{\ast}_1,h_1, \Upsilon (k^{\ast}_2,h_2,q))) & = & k^{\ast}_1 \circ h_1 \circ k^{\ast}_2 \circ h_2 (q) \\
& =& k^{\ast}_1 \circ h_1 \circ k^{\ast}_2 \circ h_1^{-1} \circ h_1 \circ h_2 (q) \\
& = & \Upsilon (((k^{\ast}_1,h_1)(k^{\ast}_2,h_2)) ,q) \, .
\end{eqnarray*}
The reader will note that in the above sequence of identities, there is no issue with the domains of definitions
of the corresponding transformations since $h, h^{-1}$ are globally defined on~$M$. There follows that
$\Upsilon$ is a local action of $G$ on $M$. To end the proof of the theorem, it remains to check that $\Upsilon$
satisfies the condition in Definition~\ref{MaximalTransformationGroup}.

Fix a point $p \in M$ and let $\{ g_i = (k^{\ast}_i ,h_i ) \}$ be a sequence of elements in $G$
such that $\{ (k^{\ast}_i ,h_i , p) \} \subset \mathcal{U}$ and $\{ (k^{\ast}_i ,h_i , p) \}$ converges to
a point in the boundary $\partial \mathcal{U}$ of $\mathcal{U} \subset G \times M \simeq K \rtimes H \times M$.
The limit point in $\partial \mathcal{U}$ of the mentioned sequence will be denoted by
$(\overline{k}^{\ast} , \overline{h} ,p)$. With this notation, we have to show that the sequence
$k^{\ast}_i \circ h_i (p)$ leaves every compact set in $M$. Assume aiming at a contradiction that this is not
the case. Thus, up to passing to a subsequence, we can assume that $\{ k^{\ast}_i \circ h_i (p) \}$ converges
to a point $\overline{p} = \overline{k}^{\ast} \circ \overline{h} (p)$ which lies
in some compact part $\overline{V} \subset M$.

It follows that $\overline{h}^{-1} (\overline{p}) = \overline{h}^{-1} \circ \overline{k}^{\ast} \circ \overline{h} (p)$
lies in a compact part $\overline{h}^{-1} (\overline{V})$ of $M$. Moreover, we have
$$
\overline{h}^{-1} (\overline{p}) = \lim_{i\rightarrow \infty} h_i^{-1} \circ k^{\ast}_i \circ h_i (p) =
\lim_{i\rightarrow \infty} k_i (p) \, ,
$$
where $k_i$ is such that $h_i \circ k_i = k^{\ast}_i \circ h_i$. Clearly $k_i$ converges to the point
$\overline{k} = \overline{h}^{-1} \circ \overline{k}^{\ast} \circ \overline{h}$. To finish
the proof of the theorem, it suffices to see that $\overline{k}$ lies in the boundary of $\Omega_p$.
This follows from the above claim since $(\overline{h} , \overline{k}) \simeq (\overline{k}^{\ast},
\overline{h})$ lies in $\partial \mathcal{U}$. Theorem~B is proved.
\end{proof}

Corollary~C is pretty much an immediate consequence of Theorem~B.

\begin{proof}[Proof of Corollary C]
The first assertion in Corollary~C is nothing but a particular case of Theorem~B. To show that the second assertion
is also a particular case of Theorem~B, consider the corresponding setting where $Z$ is a holomorphic vector field
on a complex manifold $M$. Recall that $X$ and $Y$ denote the real vector fields on $M$ which are respectively induced by the real (local)
flow of $Z$ and by the purely imaginary (local) flow of $Z$. Thus $X$ and $Y$ are real vector fields on $M$ whose local
flows are constituted by holomorphic maps. More importantly, $X$ and $Y$ clearly commute. Finally, they are both semicomplete
since every real vector field is automatically semicomplete, i.e. their integral curves always admit a maximal interval of definition
in $\R$. Thus if we assume that one of the vector fields $X$ or $Y$ is actually $\R$-complete, the situation in question
becomes again a particular case of Theorem~B for real Lie groups and in the smooth category. Hence there follows that
the local $\R^2$-action arising from the (real)
flows of $X$ and $Y$ is, indeed, maximal. However, this $\R^2$-action is by construction identified
with the (local) $\C$-flow of $Z$. In other words, the complex local flow of $Z$ has a maximal domain of
definition (i.e. it is semi-global) and hence
$Z$ must be semicomplete. This ends the proof of Corollary~C.
\end{proof}

To close this section let us state and prove Proposition~\ref{Maximal-X-Y-TheoremA}. Consider a pair of holomorphic
vector fields $X$ and $Y$ as in~(\ref{ThePair-XY-Introduction}) which are defined on a neighborhood of the origin
in $\C^2$. It is immediate to check that the vector field $X$ is semicomplete.
As to the vector field $Y$, its form is such that $[X,Y]=0$. The semicomplete character of $Y$, and more generally,
the univalent character of the Lie algebra generated by $X$ and $Y$ on a neighborhood of $(0,0) \in \C^2$ depends,
however, of the precise form of the functions $g$ and $f$.

Recall that $g$ is holomorphic whereas $f$ is allowed to be meromorphic. Nonetheless the order of the pole
of $f$ is such that the map $(x,y) \rightarrow x^a y^b f(x^ny^m)$ is holomorphic with order at least~$2$ at the origin.
Hence the linear part at the origin of the foliation associated with $Y$ is non-zero and given by
$$
-bm x \partial /\partial x + am y \partial /\partial y \, .
$$
The eigenvalues of this foliation at the origin are therefore $-b$ and $a$. We also remind the reader that
$a,b,m,n$ are all positive integers satisfying $am - bn \in \{-1,1\}$.

Finally, considering the Laurent expansion of $f$ around $0 \in \C$, let $f_0$ denote the constant term.
Now, Proposition~\ref{Maximal-X-Y-TheoremA}
reads as follows:

\begin{prop}
\label{Maximal-X-Y-TheoremA}
Let $X$ and $Y$ be as above. Assume that one of the following conditions hold:
\begin{enumerate}
  \item $g(0) \neq 0$.
  \item $g(0) = g''(0) = f_0 =0$ but  $g'(0) \neq 0$.
\end{enumerate}
Then the Lie algebra generated by $X$ and $Y$ as in~(\ref{ThePair-XY-Introduction}) is univalent
on a sufficiently small neighborhood $U$ of $(0,0) \in \C^2$.
\end{prop}

The proof of Proposition~\ref{Maximal-X-Y-TheoremA} relies on Theorem~B and on the fact that $X$ can easily be extended
to a complete vector field on a suitable rational fibration $M$.

To begin with, consider the vector field $X$ whose associated foliation $\fol_X$ is linear in the present coordinates.
Since the vector field $X$ is polynomial (actually a homogeneous multiple of a linear vector field) the foliation $\fol_X$
can be viewed as a global foliation on $\C^2$ and, indeed, on all of $\C P(2)$. The vector field $X$ is also a holomorphic
semicomplete vector field on $\C^2$, albeit its extension to $\C P(2)$ is meromorphic with the pole divisor coinciding with
the line at infinity $\Delta$.

Consider then the foliation $\fol_X$ as a global foliation defined on $\C P(2)$ and note that $\fol_X$ has exactly
$3$ singular points, namely: the origin of $\C^2$, the point $p_x$ corresponding to the intersection of the $x$-axis
with $\Delta$, and the point $p_y$ corresponding to the intersection of the $y$-axis with $\Delta$. Furthermore $\Delta$
is invariant by $\fol_X$.

The singular point $p_0$ is in the Siegel domain, while $p_x$ and $p_y$ are dicritical singular points
(with non-zero eigenvalues) for $\fol_X$. Consider the blow-up of $\C P(2)$ at $p_x$ and denote
by $\Delta_{1,x}$ the resulting component of the exceptional divisor. The blown-up foliation
$\fol_X$ (which will still be denoted by $\fol_X$) has now two singular points in $\Delta_{1,x}$.
These two singular points are determined by the intersections of $\Delta_{1,x}$ with the transforms
of $\Delta$ and of the $x$-axis. One of these singular points is (linearizable and) belongs to the Siegel domain
while the other is again dicritical with non-zero eigenvalues. We then repeat the procedure by blowing-up the
{\it dicritical singular point}. The structure keeps repeating itself until we obtain a divisor ($-1$ rational curve) $\Delta_{x}$
which {\it is not invariant} by the transformed foliation. The reader will note that this non-invariant (also called ``dicritical'')
$-1$ rational curve is everywhere transverse to the corresponding transform of the foliation $\fol_X$.

An analogous sequence of blowing-ups is also performed starting at the point $p_y$. Once dicritical singular points
no longer exist, we obtain a rational fibration $\mathcal{P}$ on the corresponding manifold $N$ (obtained by means of the indicated sequence of
blow-ups of $\C P(2)$). The corresponding non-invariant divisor ($-1$ rational curve) will similarly be denoted by
$\Delta_y$. The fibers of this rational fibration can naturally be identified with the leaves of $\fol_X$.
In particular $X$ is complete on $N$ since the fibers are compact. Indeed, the transform of $X$ on $N$ is such that $X$ vanishes
with order~$2$ over one of the curves $\Delta_{x}$, $\Delta_{y}$ while it is regular non-zero over the other. The flow of $X$
on a chosen fiber is therefore conjugate to the ``parabolic'' flow of $x^2 \partial /\partial x$ on $\C P(1)$.
Without loss of generality, the curve where $X$ vanishes with order~$2$ can be assumed to coincide with $\Delta_x$.

The fiber of $\mathcal{P}$ passing through the origin of $\C^2$ (naturally identified with a point in $N$)  is singular:
it consists of a sequence of rational curves containing the (initial) cartesian axes $\{ y=0\}$ and $\{ x=0\}$.
Next, consider
a small bidisc $B_{\varepsilon}$ of radii $\varepsilon >0$ around $(0,0) \in \C^2$ and denote by
$M_{\varepsilon} \subset N$ the saturated of $B_{\varepsilon}$ by the fibers of the above constructed rational fibration.
If $\varepsilon >0$ is small enough, then $X$ is holomorphic on $M_{\varepsilon}$. Also the restriction of $X$ to
$M_{\varepsilon}$ is still complete on $M_{\varepsilon}$.

Up to reducing $\varepsilon >0$, we can assume that $Y$ is defined on $B_{\varepsilon}$. In turn,
$B_{\varepsilon}$ is identified with its image on $M_{\varepsilon}$. Similarly, the fibration $\mathcal{P}$ can be restricted to
$M_{\varepsilon}$.

Finally it should be pointed out that the fibers of $\mathcal{P}$ intersect $B_{\varepsilon}$ in connected sets. Indeed, the flow
of $X$ being ``parabolic'' on each fiber of $\mathcal{P}$, $M_{\varepsilon}$ is essentially a globalization
of $X$ on $B_{\varepsilon}$: we say ``essentially'' because the actual globalization with be the complement in
$M_{\varepsilon}$ of the rational curve $\Delta_{x}$ (recall that $X$ vanishes with order~$2$ on $\Delta_x$).

Setting $D = \Delta_x \cap M_{\varepsilon}$, there follows that the integral curves of $X$ (fibers of $\mathcal{P}$)
intersect $D$ at a single point. Thus we can use $D \subset \Delta_x$ as basis for $\mathcal{P}$ (notation:
$\mathcal{P} : M_{\varepsilon} \rightarrow D$).

\begin{lema}
\label{Tired-1}
The vector field $Y$ admits a meromorphic (possibly holomorphic)
extension to $M_{\varepsilon}$ with poles contained in $D = \Delta_x \cap M_{\varepsilon}$. Moreover $D$
is invariant by the foliation associated with $Y$.
\end{lema}

\begin{proof}
Whereas the vector field $Y$ is only defined on $B_{\varepsilon}$, it can be considered as
globally defined on the fiber of $\mathcal{P}$ through $(0,0)$: this is clear when $Y$
vanishes identically on the coordinate axes. On the other hand, if the restriction of $Y$ to, say the $x$-axis, is not
identically zero (i.e. $g(0) \neq 0$), then it is given by $x \partial /\partial x$ and the claim follows again.
Next we observe that $Y$ has a holomorphic extension, still denoted by $Y$ to all of $M_{\varepsilon} \setminus \Delta_x$. Indeed,
to define $Y$ at a point $p \in M_{\varepsilon} \setminus \Delta_x$ which does not lie in the fiber of
$\mathcal{P}$ through $(0,0) \in \C^2$, we proceed as follows: choose $t \in \C$ such that $\phi (t,p)$ belongs to
$B_{\varepsilon}$, where $\phi$ stands for the flow of $X$ on $M_{\varepsilon}$. Consider also
the corresponding diffeomorphism $\phi^t : M_{\varepsilon} \rightarrow M_{\varepsilon}$. Then set $Y (p) =
D_{\phi^t (p)}\phi^{-t} . Y(\phi^t (p))$. The fact that the (semi-global) flow of $X$ in $B_{\varepsilon}$ preserves $Y$
combines with the connectedness of the intersections of the fibers of $\calp$ with $B_{\varepsilon}$ to ensure that
the vector $Y(p)$ is unambiguously defined so that $Y$ has the desired holomorphic extension to
$M_{\varepsilon} \setminus \Delta_x$.

To finish the proof, it only remains to consider the extension of $Y$ to $D$. Since $D$ is not contained in
the fiber $\mathcal{P}^{-1} (0)$, we denote by $P$ the intersection point $D \cap \mathcal{P}^{-1} (0)$.
Around $P$, there are coordinates $(z,w)$ such that $\{ z=0\} \subset D$,
$\{ w=0 \} \subset \mathcal{P}^{-1} (0)$, and where $X$ is locally given by $wz^2 \partial /\partial z$. This
follows directly from the construction of $\mathcal{P}$.
Indeed, essentially it suffices to keep in mind that $\Delta_x$ is a $-1$-curve arising from the blow up of a radial singular
point for the foliation associated with $X$ (so that the transform of the foliation associated with $X$ will be transverse
to $\Delta_x$). To consider the vector field $Y$, we then proceed as follows. First, it is enough to consider
the vector field $Y' = Y /g(x^ny^m)$. Note that $Y'$ still commutes with $X$. Furthermore, since around $p$ the first
integral of $X$ becomes the function $w$, it follows that $Y' = Y /g(w)$ and hence it suffices to prove that $Y'$ has a meromorphic
extension as indicated.

Now note that $Y'$ is given by the sum of the linear vector field $Y'_{\rm lin} = -bmx \partial /\partial x + am y
\partial /\partial y$ plus the vector field $fX/m$. The linear vector field can naturally be transformed under
the blow-ups used in the construction of $\mathcal{P}$. It turns out that its expression in $(z,w)$
coordinates is $Y'_{\rm lin} = - z \partial /\partial z +  w \partial /\partial w$ (up to a multiplicative constant).
In turn, $fX/m$ becomes $(f(w) wz^2)/m  \partial /\partial z$ since $f$ is function of the first integral $x^n y^m$ and
$X$ is transformed in $wz^2 \partial /\partial z$. Therefore, in $(z,w)$-coordinates, we have
\begin{equation}
Y' = -z (1 + wzf(w)/m) \partial /\partial z - w \partial /\partial w \, . \label{Y'-thevectorfield-11}
\end{equation}
Since $f$ is meromorphic, the desired extension of $Y'$ and $Y$ follows. Moreover, whether or not $f$ is meromorphic,
the foliation associated with $Y$ leaves the axis $\{ z=0\} \subset D$ as can be seen by multiplying $Y'$ by the power of $w$
corresponding to the order of the pole of $f$ minus~$1$. The lemma is proved.
\end{proof}

Owing to Lemma~\ref{Tired-1}, the vector field $Y$ defines a singular holomorphic foliation $\fol_Y$ on all of $M_{\varepsilon}$.
Furthermore, $D$ is invariant by $\fol_Y$. In particular, the
local holonomy map of $D$ with respect to $\fol_Y$ can then be considered on a neighborhood
of $P$. Recall that $f_0$ is the constant term in the Laurent expansion of $f$ around $0 \in \C$.

\begin{lema}
\label{Tired-2}
The local holonomy map of $D$ with respect to $\fol_Y$ coincides with the identity if and only if $f_0 =0$.
\end{lema}

\begin{proof}
Since the statement depends only on the foliation $\fol_Y$ (and not on $Y$), we can consider the vector field
$Y'$ inducing the foliation $\fol_Y$ around~$P$. In the above coordinates $(z,w)$, $Y'$ is as in~(\ref{Y'-thevectorfield-11}).
To compute the holonomy map in question, consider the loop $c (k) = e^{2\pi i k}$, $k \in [0,1]$,
contained in $D \subset \{ z=0\}$.
We restrict $w$ to this loop and set $w = e^{2\pi i k}$ so that $w' = 2\pi i e^{2\pi i k} = 2\pi i w$.
The lift $(z(k), w(k))$ of $c$ in a leaf of $\fol_Y$ near $D$
is such that
\begin{eqnarray*}
\frac{dz}{dk} = \frac{dz}{dw} \frac{dw}{dk} & =  & \frac{-z(1 + wzf(w)/m)}{w} \, \ 2\pi i e^{2\pi i k} \\
& = & -2\pi i z(1 + wzf(w)/m) \, .
\end{eqnarray*}
To solve the above equation let $z (k) = \alpha (k) e^{-2\pi i k}$ so that
$$
\alpha' = -2\pi i \alpha^2 \frac{f (e^{-2\pi i k})}{m} \, .
$$
Integration of $\alpha'/\alpha^2$ from $k=0$ to $k=1$ yields
$$
-\frac{1}{\alpha (1)} + \frac{1}{\alpha (0)} = -2\pi i f_0 \, .
$$
Since $\alpha (0) = z(0)$ and $\alpha (1) = z(1)$, we conclude that the holonomy map associated with $D$
is given by the map
$$
z \longmapsto \frac{z}{1+2\pi i f_0z} \, .
$$
The above map coincides with the identity if and only if $f_0 =0$. Otherwise, it consists of a germ of parabolic
map with infinite order. The lemma is proved.
\end{proof}

\begin{proof}[Proof of Proposition~\ref{Maximal-X-Y-TheoremA}]
Owing to Theorem~B (or item~($\imath )$ in Corollary~C),
the proof of the proposition is reduced to showing that $Y$ is semicomplete on $M_{\varepsilon}$,
up to reducing $\varepsilon$. Indeed, if $Y$ is semicomplete as indicated, then Theorem~B ensures
that the Lie algebra generated by $X$ and $Y$ is univalent
on $M_{\varepsilon}$. In particular its restriction to $B_{\varepsilon} \subset M_{\varepsilon}$ is univalent as well.

We will then prove that $Y$ is semicomplete on $M_{\varepsilon}$ provided that the
functions $g$ and $f$ satisfy one of the conditions in the statement.
For this, recall that the
vector field $Y$ must preserve the fibration $\mathcal{P} : M_{\varepsilon} \rightarrow D$ and thus
it projects to a well defined one-dimensional vector field $Z$ on $D$. In fact, the discussion conducted
in the proof of Lemma~\ref{Tired-1} shows that
$$
Z = w g(w) \partial /\partial w \, .
$$

On the other hand, the foliation $\fol_Y$ is transverse to the fibers
of $\mathcal{P}$ away from the singular fiber. Unless explicit mention in contrary, all leaves of $\fol_Y$ considered
in the sequel are assumed not to be contained in $\mathcal{P}^{-1} (0)$. If $L$ is one such leaf, then the restriction
of $\mathcal{P}$ to $L$ yields a local diffeomorphism from $L$ to $D$.

\vspace{0.1cm}

\noindent {\it Claim}. The restriction $\mathcal{P}_L$ of $\mathcal{P}$ to $L$ as before is one-to-one if and only
if $f_0 =0$.

\noindent {\it Proof of the Claim}. The statement amounts to showing that the monodromy map $h$ of $\fol_Y$ with respect
to the fibration $\mathcal{P}$ coincides with the identity. Owing to the fact that $\fol_Y$ is (globally)
transverse to the fibration $\mathcal{P}$, there follows that $h$ is naturally identified with
an automorphism of a rational curve. However, since $D$ is invariant by $\fol_Y$, $D$ is naturally associated with
a fixed point of $h$. Moreover, the germ of $h$ at this fixed point coincides with the germ of the {\it local}\,
holonomy map of $D$ with respect to $\fol_Y$. In view of Lemma~\ref{Tired-2}, the mentioned germ coincides with the
identity if and only if $f_0 =0$ and this establishes Claim~1.\qed

Let us now assume aiming at a contradiction that $Y$ is not semicomplete on $M_{\varepsilon} \setminus \Delta_x$.
Denote by $\fol_Y$ the foliation associated to $Y$. Each leaf $L$ of $\fol_Y$ (not contained in $\mathcal{P}^{-1} (0)$)
is endowed with an abelian form $dT_L$ defined on $L$ by the pairing $dT_L . Y =1$. The abelian form $dT_L$ is called the
time-form induced by $Y$ on $L$. Now, if $Y$ is not semicomplete,
then there exists a leaf $L$ and an open (embedded) path $c : [0,1] \rightarrow L$ such that the integral of $dT_L$ over
$c$ equals zero. The path $c$ projects on $D$ as a path $\mathcal{P} (c) (t)$ that either is open or is a loop winding around $P \in D$ a
certain number of times (strictly different from zero). By construction, we have
$$
\int_c dT_L = \int_{\mathcal{P} (c)} \frac{dw}{wg(w)} \, .
$$

Assume now that $g(0) \neq 0$. Then the integral of $dw/wg(w)$ over $\mathcal{P} (c)$ cannot be equal to zero whether
$\mathcal{P} (c)$ is an open path or a loop with winding number different from zero. The resulting contradiction then implies
that $Y$ is semicomplete on $M_{\varepsilon}$ as desired.

Assume now that $g(0) = g''(0) =0$, $g'(0) \neq 0$ and $f_0 =0$. The vector field $Z$ is then conjugate to
$w^2 \partial /\partial w$ since its order at $0 \in \C$ is~$2$ while its residue is equal to zero. Hence the only possibility
of having
$$
\int_{\mathcal{P} (c)} \frac{dw}{wg(w)} = 0
$$
occurs when $\mathcal{P} (c)$ is a loop, possibly winding around $0 \in \C$. However, if $\mathcal{P} (c)$ is a loop
so must be $c$: since $f_0 =0$, the restriction of $\mathcal{P}$ to $L$ is injective according to the claim.
The proof of the proposition is completed.
\end{proof}

Proposition~\ref{Maximal-X-Y-TheoremA} provides sufficient conditions for the Lie algebra generated
by $X$ and $Y$ to be univalent. It is natural to wonder if these conditions are also necessary. The remainder
of this section will be devoted to prove that this is, in fact, the case. The corresponding results, however,
will not be used anywhere else in this paper.

In the sequel we always assume that $g(0) =0$ otherwise there is nothing to be proved. The condition $g'(0) \neq 0$
is therefore necessary. The assertion follows from observing that $Y$ is not semicomplete on a neighborhood of
$(0,0) \in \C^2$ if the order of $g$ at $0 \in \C$ is $l \geq 2$. Indeed, note that the first non-zero homogeneous
component of the Taylor series of $Y$ around $(0,0)$ is the vector field
$$
m \, (x^ny^m)^l [-bx \partial /\partial x + ay \partial /\partial y ] \, .
$$
This vector field must be semicomplete provided that $Y$ is semicomplete owing to the fact that the space of
semicomplete vector fields is closed under convergence on compact parts, see \cite{JR}. On the other hand, the
condition for this vector field to be semicomplete is to have $l(-bn+am)  = \pm 1$. However, since $am-bn=1$,
this cannot happen for $l \geq 2$.

In view of the preceding we must have $g'(0) \neq 0$ whenever $g(0) =0$. In the sequel, we assume that these two conditions
are satisfied. Next, we consider again the vector field
$Z = w g(w) \partial /\partial w$ induced by $Y$ in the leaf space of $X$. The order of $Z$ at $0 \in \C$ is equal
to~$2$. If the residue of the vector field is not zero (i.e. if $g''(0) \neq 0$), then there exists an open path
(``near a loop'') over which
the integral of the $1$-form $dw/wg(w)$ vanishes (see \cite{JR}). In other words, $Z$ is not semicomplete. Moreover,
whenever the mentioned path can be lifted to a leaf of $\fol_Y$, the restriction of $Y$ to the leaf in question
will not be semicomplete as well. On the other hand, if $g''(0) =0$, then the integral $dw/wg(w)$ over any small loop
winding about $0 \in \C$ is equal to zero. In addition, if $f_0 \neq 0$, then whenever a lift of the mentioned loop in a leaf of
$\fol_Y$ exists, the restriction of $Y$ to this leaf will not be semicomplete. Summarizing, to prove that
the conditions $g''(0) =0$ and $f_0 =0$ are also necessary for $Y$ to be semicomplete (and hence for the Lie algebra
generated by $X$ and $Y$ to be univalent), it suffices to prove the lemma below:

\begin{lema}
Fixed $\varepsilon >0$, there is a decreasing sequence $\{ \delta_j \}$ converging to~$0$ ($\delta_j >0$ for all $j$)
such that the loop $C(k) = \delta_j e^{2\pi i k}$, $k \in [0,1]$, contained in $D$ can be lifted in a leaf of $\fol_Y$
so that the lifted path is entirely contained in $B_{\varepsilon}$.
\end{lema}

\begin{proof}
We begin with a simple remark to be used in the course of the proof. Denoting by $r$ the order of the pole of $f$,
the estimate $b \geq mr$ must hold since the map
$(x,y) \rightarrow x^a y^b f(x^ny^m)$ is holomorphic.

Next, it is convenient to use slightly different (albeit essentially equivalent) coordinates $(z,w)$. Indeed, we consider
the singular map $H(z,w) = (z^m, w/z^m) = (x,y)$. We can assume $H$ to be defined on the set
\[
V = \{ (z,w) \in \C^2 \; , \; 0 < |z| < \sqrt[m]{\varepsilon} \;  \, {\rm and} \; \, |z| > \sqrt[m]{\frac{|w|}{\varepsilon}} \} \, .
\]
The reader will check that $H (V) \subset B_{\varepsilon}$.

Let $\mathcal{X}$ (resp. $\mathcal{Y}$) denote the vector field $X$ (resp. $Y$) in the (``singular'') coordinates $(z,w)$.
A direct computation yields
\[
\mathcal{X} = z^2 w^b \frac{\partial}{\partial z} \,
\]
and
\[
\mathcal{Y} = g(w^m) \left[ z \left( -b + \frac{1}{m} z w^b f(w^m) \right) \frac{\partial}{\partial z} +
w \frac{\partial}{\partial w} \right] \, .
\]
Note that the (local) leaf space of $X$ in $B_{\varepsilon}$ is given by $x^n y^m = {\rm cte}$ which, in our case,
means $w^m = {\rm cte}$. The leaf space in question is thus identified with the quotient of the $w$-axis by the rotation
of order~$m$.

With the preceding notation, to prove the lemma it suffices to show that for $\delta >0$ arbitrarily small,
the loop $w(k) = \delta e^{2\pi i k}$, for $k \in [0,1]$ can be lifted in some leaf of the foliation associated with
$\mathcal{Y}$ on its domain of definition~$V$.

We intend to check is the loop given in the coordinate $w$ by $w(k) = \delta e^{2\pi i k}$, for $k \in [0,1]$ can be lifted along any
leaf of the vector field $\mathcal{Y}$ on its domain of definition. To begin with, we set
\begin{eqnarray*}
\frac{dz}{dk} = \frac{dz}{dw} \, \frac{dw}{dk} & = & \frac{z \left( -b + \frac{1}{m} z w^b f(w^m) \right)}{w} 2\pi i w \\
& = & 2\pi i z \left( -b + \frac{1}{m} z \delta^b e^{2\pi i bk} f \left(\delta^m e^{2\pi i mk} \right) \right)
\end{eqnarray*}
Therefore
\begin{equation}\label{eq_diff_z_new}
\frac{dz}{dk} + 2\pi i b z = \frac{2\pi i}{m} z \delta^b e^{2\pi i bk} f \left(\delta^m e^{2\pi i mk} \right) \, .
\end{equation}
Let again $z(k) = c(k) e^{-2\pi i b k}$ so that $\vert z(k) \vert = \vert c(k) \vert$ and $c(k)$ satisfies the
equation
\[
\frac{dc}{dk} = \frac{2\pi i}{m} c^2 \delta^b f\left( \delta^m e^{2\pi i mk} \right) \, .
\]
By integrating $c'/c^2$ from $0$ to $k$, we obtain
\[
\frac{1}{c(k)} - \frac{1}{c (0)} = \delta^b \int_0^k \frac{2\pi i}{m} f\left( \delta^m e^{2\pi i mt} \right) dt \, .
\]
Thus
\[
c(k) = \frac{c(0)}{1 - c(0) \delta^b \int_0^k \frac{2\pi i}{m} f\left( \delta^m e^{2\pi i mt} \right) dt}
\]

Next note that $\vert f\left( \delta^m e^{2\pi i mt} \right) \vert$
has order $\delta^{-mr}$ for $\delta$ small. Thus $\delta^b \vert f\left( \delta^m e^{2\pi i mt} \right) \vert$
is bounded by a uniform constant when $\delta \rightarrow 0$. Hence, there is a constant $C_1$ (uniform as $\delta \rightarrow 0$)
such that
$$
\left| \int_0^k \delta^b \frac{2\pi i}{m} f\left( \delta^m e^{2\pi i m s} \right) ds \right| \leq  C_1 k \, .
$$
In turn, it follows the existence of a uniform constant $\beta >0$ such that
\begin{equation}
\vert c(k) -c(0) \vert \leq \beta \vert c(0) \vert^2 k \, . \label{liftingpaths-1}
\end{equation}

We are interested in the above estimate for $0 \leq k \leq 1$. In fact,
it is sufficient to consider the case $k=1$: if it can be proved that $z(1)$ belongs to $V$, it becomes clear
from the above estimate that $z(k)$ is entirely contained in $V$ for $k \in [0,1]$.

Recall that $V$ is defined by $|z| < \sqrt[m]{\varepsilon}$ and $|z| > \sqrt[m]{|w|/\varepsilon}$ and $\varepsilon$ is
fixed. On the other hand, we have
$\vert w (k) \vert = \delta$ for all $k \in [0,1]$. Take $\vert z(0) \vert = \vert c(0) \vert = \delta^{1/(m+\tau)}$
for some $\tau >0$ small. Thus $(z(0),w(0))$ lies in $V$ so long $\delta$ is small enough. Similarly equation~(\ref{liftingpaths-1})
ensures that
$$
\vert c(0) \vert \; (\, 1 - \beta k \vert c(0)\vert \, )  \leq \vert c(k) \vert \leq
\vert c(0) \vert \; (\, 1 + \beta k \vert c(0) \vert \, ) \, .
$$
It is now clear that $c(1)$ (and hence $c(k)$ for $k \in [0,1]$) satisfies $\vert c(1) \vert = \vert z(1) \vert <
\sqrt[m]{\varepsilon}$ provided that $\delta$ is small enough. Similarly, again for sufficiently small $\delta >0$,
we also have $\vert c(1) \vert^m = \vert z(1) \vert^m > \delta/ \varepsilon = |w|/\varepsilon$. This proves that the
lifted path $(z(k) ,w(k))$ remains in $V$ for $k \in [0,1]$ and ends the proof  of the lemma.
\end{proof}


\section{Globalization problem and leaf space for commuting vector fields}\label{Sec_realization}

Let us then consider the vector fields $X, \, Y$ defined on a neighborhood $U$ of the origin of $\C^2$ by
\[
X =  x^a y^b\left[ mx \frac{\partial}{\partial x} - ny \frac{\partial}{\partial y} \right]
\]
and by
\[
Y =  g(x^n y^m) \left[ x (-bm + x^a y^b f(x^ny^m)) \frac{\partial}{\partial x} + y \left(a m - \frac{n}{m} x^a y^b f(x^n y^m)\right)
\frac{\partial}{\partial y} \right]
\]
for some holomorphic function $g$ and some meromorphic function $f$. The function $f$, however, is such that
the map $(x,y) \rightarrow x^a y^b f(x^ny^m)$ is holomorphic with order at least~$1$ at the origin.
Finally $a,b,m,n$ are positive integers satisfying $am - bn \in \{-1,1\}$. The upshot being that
$X$ and $Y$ commute.

The $2$-dimensional
foliation spanned by $X$ and $Y$ on $U$ will be denoted by $\mcf$. This foliation is rather trivial since the open set
$U \setminus ( \{ x=0\} \cup \{ y=0\})$ is a leaf of $\mcf$. Nonetheless, following the discussion in
Section~$2$, we should rather consider the foliation $\overline{\mcf}$ on $\C^2 \times U$ spanned by
the commuting vector fields $\overline{X} = \partial /\partial t + X$ and $\overline{Y} = \partial /\partial s + Y$, i.e.
\[
\overline{X} = \frac{\partial}{\partial t} + x^a y^b\left[ mx \frac{\partial}{\partial x} - ny \frac{\partial}{\partial y} \right]
\]
and
\[
\overline{Y} = \frac{\partial}{\partial s} + g(x^n y^m) \left[ x (-bm + x^a y^b f(x^ny^m)) \frac{\partial}{\partial x} +
y \left(a m - \frac{n}{m} x^a y^b f(x^n y^m)\right) \frac{\partial}{\partial y} \right] \, .
\]
The (regular) foliation $\overline{\mcf}$ has leaves of complex dimension~$2$.

The objective of this section is to study the leaf space of $\overline{\mcf}$ so as to prove it is Hausdorff;
cf. Theorem~\ref{FinalTheoremCommutingVF}. By combining this theorem and Proposition~\ref{Maximal-X-Y-TheoremA},
Theorem~A in the Introduction follows at once. Yet, Theorem~\ref{FinalTheoremCommutingVF} may have some interest on its own,
even when the Lie algebra generated by $X$ and $Y$ is not univalent, since the Hausdorff character of leaf space is a
very rare and clearly important phenomenon.

The study of the leaf space of $\overline{\mcf}$ should essentially be divided in three cases, namely:

\smallbreak

{\bf Case 1}: Case where $ab\neq 0$;

{\bf Case 2}: Case where $a=0$ (this assumption immediately implies that $b=1$ and $n=1$);

{\bf Case 3}: Case where $b=0$ (analogously to the previous case, if $b=0$ then $a=1$ and $m=1$).

\smallbreak

It should be mentioned that, no matter the case we are considering, $\C^2 \times U$ admits a partition on four disjoint
manifolds that are invariant under the $2$-dimensional foliation $\overline{\mcf}$ and will be denoted by
$S_0, S_x, S_y$ and $S_{x,y}$. Namely, we have:
\begin{align*}
&S_0 = \{(t,s,x,y) \in \C^2 \times U: x=y=0\} \\
&S_x = \{(t,s,x,y) \in \C^2 \times U: x \ne 0, \, y=0\} \\
&S_y = \{(t,s,x,y) \in \C^2 \times U: x=0, \, y \ne 0\} \\
&S_{x,y} = \{(t,s,x,y) \in \C^2 \times U: xy \ne 0\} \, .
\end{align*}
The leaf spaces arising from the restrictions of $\overline{\mcf}$ to each one of these invariant manifolds
are summarized in Table~\ref{table: spaceofleaves}. It actually suffices to describe the corresponding spaces
in Case~$1$ since the remaining cases follow analogously. So let us assume in what follows that $ab \neq 0$.

\bigbreak

{\bf 1. Space of leaves on $S_0$}

\smallbreak

First of all, it should be noted that $S_0$ is itself a leaf of the foliation $\overline{\mcf}$ regardless
of the values of $m, \, n, \, a$ and $b$ and of the functions $f$ and $g$. In fact, the restriction of
$\overline{X}$ and $\overline{Y}$ to this manifold is simply given by
\[
\overline{X}|_{S_0} = \frac{\partial}{\partial t}   \quad \mbox{ and } \quad  \overline{Y}|_{S_0} = \frac{\partial}{\partial s} \, .
\]
Thus $S_0$ corresponds to a single point in the space of leaves of $\overline{\mcf}$.

\bigbreak

{\bf 2. Space of leaves on $S_y$ (the case of $S_x$ being analogous)}

\smallbreak

Both $S_x$ and $S_y$ are $3$-dimensional manifolds equipped with $2$-dimensional foliations, namely the restrictions
of $\overline{\mcf}$ to $S_x$ and to $S_y$, respectively. The space of leaves of $\overline{\mcf}$ on $S_y$ (resp. $S_x$)
corresponds then to a $1$-dimensional manifold. Let us then describe this manifold.

The restrictions of the vector fields $\overline{X}$ and $\overline{Y}$ to $S_y$ are respectively given by
\[
\overline{X}|_{S_y}  =   \frac{\partial}{\partial t}  \quad \mbox{ and } \quad
\overline{Y}|_{S_y}  =  \frac{\partial}{\partial s} + a m g(0) y \frac{\partial}{\partial y} \, .
\]
Let $\overline{\mcf}|_{S_y}$ denote the foliation spanned by these two vector fields on $\C^2 \times V$, where $V = \{
y \in \C^{\ast}: (0,y) \in U \}$. Consider the fibration of $S_y \simeq \C^2 \times V$ with base $V$, whose fiber map
is given, in coordinates $(t,s,y)$ by $\pi_y (t,s,y) = y$.

In the simpler case where $g(0)=0$, the foliation $\overline{\mcf}|_{S_y}$ is induced by the vector fields $\partial /\partial t$
and $\partial /\partial s$ and, consequently, the fibers of the above mentioned fibration are the leaves of the given foliation.
The leaf space in this case is simply the base of the fibration, namely the punctured disc $\mathbb{D}^{\ast}$.

As for the case where $g(0) \ne 0$, the projection of every single leaf of $\overline{\mcf}|_{S_y}$ covers the base $V$. More
precisely, fixed $\varepsilon$ arbitrarily small and denoting by $\Sigma_{\varepsilon}$ the fiber above $y = \varepsilon$,
every leaf of $\overline{\mcf}|_{S_y}$ intersects $\Sigma_{\varepsilon}$ transversely. Furthermore the projection
$\pi_y$ restricted to a leaf of $\overline{\mcf}|_{S_y}$
provides a covering map onto the base. In turn, the fundamental
group of the base has a single generator $\sigma$. To describe
the space of leaves of $\overline{\mcf}$ on $S_y$, we have to compute the monodromy map associated with $\sigma$.
The mentioned monodromy map is entirely determined by $\overline{Y}|_{S_y}$ since the $1$-dimensional leaves defined by
$\overline{X}|_{S_y}$ are contained in the fibers. Up to a homothety, we can assume that $\sigma$ is given by
\[
y(k) = \mathrm{e}^{2\pi i k}, \,  k\in [0,1] \, .
\]
Thus, we have
\[
\frac{d s}{d k} = \frac{d s}{d y}\frac{d y}{d k} = \frac{1}{ a m g(0) y} 2 \pi i y= \frac{2\pi i}{a m g(0)} \, .
\]
Integrating the above differential equation, we obtain $s(k) = c + \frac{2\pi i}{a m g(0)} k$, with $c \in \C$.
Looking at $k=0$ and at $k=1$, we conclude that the monodromy map with respect to $\overline{Y}$ is given by
\[
h_y: \, \, s \longmapsto s + \frac{2\pi i}{a m g(0)} \, .
\]
Thus, the leaf space over $S_y$ in the particular case where $g(0) \ne 0$ is the quotient of $\mathbb{C}$ under the
translation above, that is,
\[
\Large{\sfrac{\mathbb{C}}{<s \mapsto s+\sfrac{2\pi i}{a m g(0)}>}} \, .
\]

\bigbreak

{\bf 3. Space of leaves over $S_{x,y}$}

\smallbreak

Next we will look at the restriction of the vector fields $\overline{X}$ and $\overline{Y}$ to $S_{x,y}$, namely
\begin{eqnarray*}
\overline{X}|_{S_{x,y}} & = &  \frac{\partial}{\partial t} + x^a y^b \left[ mx \frac{\partial}{\partial x} -
n y  \frac{\partial}{\partial y}\right] \medskip\\
\overline{Y}|_{S_{x,y}} & = & \frac{\partial}{\partial s} + g(x^ny^m) \left[
x ( -bm + x^ay^b f(x^ny^m)) \frac{\partial}{\partial x} + y( am-\frac{n}{m} x^ay^b f(x^ny^m) \frac{\partial}{\partial y}\right]
\end{eqnarray*}
To abridge notation, let $V^{\ast}$ stand for $U \setminus (\{x=0\} \cup \{y=0\})$ in the sequel.
Consider now the fibration of $S_{x,y} \simeq \C^2 \times V^{\ast}$ with base $V^{\ast}$
whose projection is given in coordinates $(t,s,x,y)$ by $\pi_{x,y} (t,s,x,y) = (x,y)$. The fibers of
this fibration have dimension~$2$ as do the leaves of the restriction of $\overline{\mcf}$ to $S_{x,y}$. Furthermore,
the leaves of the restriction of $\mcf$ to $S_{x,y}$ are transverse to the fibers of $\pi_{x,y}$ and, again,
the restriction of $\pi_{x,y}$ to any one of these leaves yields a covering map onto $V^{\ast}$.

The base $V^{\ast}$ of the above mentioned fibration  has fundamental group generated by two loops
$\sigma_1$ and $\sigma_2$ which can be chosen (up to homothety) as
\[
\sigma_1(k) = (\mathrm{e}^{2\pi i m k},\mathrm{e}^{-2\pi i n k})  \quad \mbox{ and } \quad   \sigma_2(k)  = (\mathrm{e}^{-2\pi i b k},
\mathrm{e}^{2\pi i a k}), \quad k\in[0,1] \, .
\]

Let us first compute the monodromy map of $\overline{\mcf}$ associated with $\sigma_1$. For this we should note that $\sigma_1$ is
tangent to $X$ so that the loop $\sigma_1$ can be lifted along the leaves of $\overline{X}$. Recalling that for $\sigma_1$ we
have $x(k) = \mathrm{e}^{2\pi i m k}$ and $y(k)=\mathrm{e}^{-2\pi i n k}$, $k \in [0,1]$, there follows that the lift of
$\sigma_1$ satisfies
\begin{align*}
\frac{dt}{dk} &= \frac{dt}{dx}\frac{dx}{dk} \left( = \frac{dt}{dy}\frac{dy}{dk} \right)
= \frac{1}{mx^{a+1}y^b} 2 \pi i m x \\
&= \frac{2 \pi i}{x^a y^b}
= \frac{2 \pi i}{\mathrm{e}^{2\pi i (a m - bn) k} } \\
&= 2 \pi i \mathrm{e}^{\mp 2\pi i k}
\end{align*}
and
\[
\frac{ds}{dk} = \frac{ds}{dx}\frac{dx}{dk} \left( = \frac{ds}{dy}\frac{dy}{dk} \right) = 0 \, .
\]
Integrating the above system of differential equations, we obtain
$$
t(k) = c_1 \mp \mathrm{e}^{\mp 2\pi i k} \; \; \; {\rm and} \; \; \;
s(k) = c_2
$$
for some constants $c_1, c_2 \in \mathbb{C}$. Thus we have $t(1) = t(0)$ and $s(1) = s(0)$. Summarizing, the monodromy map
$\varphi_1$ arising from $\sigma_1$ reduces to the identity.

As for the monodromy map with respect to $\sigma_2$, we need to consider a combination of the vector fields $\overline{X}$ and
$\overline{Y}$ with respect to which the lift of $\sigma_2$ can be taken. Naturally, we may consider the vector field
$\overline{W}$ given as
\begin{eqnarray*}
\overline{W} & = & \frac{1}{m} g(x^ny^m) f(x^ny^m) \overline{X} - \overline{Y} \medskip\\
& = & \frac{1}{m} g(x^ny^m) f(x^ny^m) \frac{\partial}{\partial t} - \frac{\partial}{\partial s}
- m g(x^ny^m) \left[ -bx \frac{\partial}{\partial x} + ay \frac{\partial}{\partial y} \right] \, .
\end{eqnarray*}
Now, recalling that for $\sigma_2$, $x(k) = \mathrm{e}^{-2\pi i b k}$ and $y(k) = \mathrm{e}^{2\pi i a k}$, $k \in [0,1]$,
there follows that
\begin{align*}
\frac{dt}{dk} &= \frac{dt}{dx}\frac{dx}{dk} \left( = \frac{dt}{dy}\frac{dy}{dk} \right)
= \frac{\frac{1}{m} g(x^ny^m) f(x^ny^m)}{-m g(x^ny^m)(-bx)} (-2 \pi i b x) \\
&= - \frac{2\pi i f\left( \mathrm{e}^{\pm 2\pi i k} \right)}{m^2}
\end{align*}
and
\begin{align*}
\frac{ds}{dk} &= \frac{ds}{dx}\frac{dx}{dk} \left( = \frac{ds}{dy}\frac{dy}{dk} \right)
= \frac{-1}{-m g(x^ny^m)(-bx)} (-2 \pi i b x) \\
&= \frac{2\pi i}{m g\left( \mathrm{e}^{\pm 2\pi i k} \right)}
\end{align*}
In turn, the functions $f$ and $1/g$ can be expanded in Laurent series as
$$
f(z) = \sum_{p \geq p_1} f_p z^p \; \; \;
{\rm and} \; \; \; 1/g(z) = \sum_{p\geq p_2} g_p z^p \, .
$$
The above equations then become
\[
\frac{dt}{dk} = - \frac{2\pi i}{m^2} f_0 - \frac{2\pi i}{m^2} \sum_{\substack{p \geq p_1 \\ p \ne 0}} f_p \mathrm{e}^{\pm 2\pi i p k}
\]
and
\[
\frac{ds}{dk} = \frac{2\pi i}{m} g_0 + \frac{2\pi i}{m} \sum_{\substack{p \geq p_2 \\ p \ne 0}} g_p \mathrm{e}^{\pm 2\pi i p k} \, .
\]
Finally by integrating the above system of differential equations above, we have
\[
t(k) = c_1 - \frac{2\pi i}{m^2}f_0 k \mp \frac{1}{m^2} \sum_{\substack{p \geq p_1 \\ p \ne 0}}
\frac{f_p}{p} \mathrm{e}^{\pm 2\pi i p k}
\]
and
\[
s(k) = c_2 + \frac{2 \pi i}{m} g_0 k \pm \frac{1}{m}\sum_{\substack{p \geq p_2 \\ p \ne 0}} \frac{g_p}{p} \mathrm{e}^{\pm 2\pi i p k}
\]
for some constants $c_1, c_2 \in \mathbb{C}$. Therefore $t(1) = t(0) - \frac{2\pi i}{m^2}f_0$ and $s(1) = s(0) + \frac{2 \pi i}{m} g_0$.
In other words, the monodromy with respect to $\sigma_2$ is given by the translation
\[
\varphi_2 : \, \, (t,s)\longmapsto \left( t -  \frac{2 \pi i}{m^2}f_0 , s + \frac{2\pi i}{m} g_0\right) \, .
\]

The space of leaves over $S_{x,y}$ is then the quotient of $\mathbb{C}^2$ under the two translations $\varphi_1, \, \varphi_2$.
Since $\varphi_1$ reduced to the identity map, the group $G$ generated by these two translations is, clearly, a cyclic group if
not reduced to the identity. In fact, we have
that
\[
G = <\varphi_1, \, \varphi_2> = <\varphi_2> \, .
\]

\smallbreak

%
%

\begin{table}[h!]
\centering
\begin{tabular}{p{0.4\textwidth}p{0.6\textwidth}} \\ \toprule
  \multicolumn{2}{c}{Leaf space over $S_0$}  \\ \midrule
  $\{ \ast \}$ & for all $a,b \in \mathbb{N}$ \\
 \midrule
 \multicolumn{2}{c}{Leaf space over $S_y$} \\ \midrule
  $\mathbb{C} $ & $a=0$ \\
  $\mathbb{D}^\ast$	& $a\neq 0 \wedge g(0) =0$\\
  \Large{$\sfrac{\mathbb{C}}{\left< s \mapsto s + \frac{2\pi i}{a m g(0)} \right>}$} & $a\neq 0 \wedge g(0) \neq 0$  \\ \midrule
\multicolumn{2}{c}{Leaf space over $S_x$}\\ \midrule
  $\mathbb{C}$ & $b = 0$    \\ 
  $\mathbb{D}^\ast$  & $b\neq 0 \wedge g(0) = 0$ \\ 
  \Large{$\sfrac{\mathbb{C}}{\left< s \mapsto s - \frac{2\pi i}{b m g(0)} \right>}$} & $b\neq 0 \wedge g(0) \neq 0$ \\ \midrule
\multicolumn{2}{c}{Leaf space over $S_{x,y}$}  \\ \midrule
  \Large{
 $\sfrac{\mathbb{C}^2}{\left< (t,s) \mapsto (t - \frac{2\pi i}{m^2} f_0, s + \frac{2\pi i}{m} g_0) \right>}$
} &   $f_0$ and $g_0$ are, respectively, the zeroth order terms of the Laurent series of $f$ and $1/g$
\\\bottomrule
\end{tabular}
\caption{Leaf space of $\overline{\mcf}$}\label{table: spaceofleaves}
\end{table}

\bigbreak

The previous calculations automatically lead us to the following result.

\begin{prop}\label{Prop_leaf_Sxy}
The leaf space of $\overline{\mcf}$ on $S_{x,y}$ is Hausdorff independently of the values of $m, \, n, \, a, \, b$
and of the functions $g$ and $f$. \qed
\end{prop}

We are now ready to prove that the entire leaf space of $\overline{\mcf}$ on $\C^2 \times U$ is Hausdorff which
corresponds to Theorem~\ref{FinalTheoremCommutingVF} below.

\begin{teo}
\label{FinalTheoremCommutingVF}
The total leaf space of $\overline{\mcf}$ on $\C^2 \times U$ is Hausdorff for every $m, \, n, \, a, \, b$ and
functions $g$ and $f$ as above.
\end{teo}

Here it is interesting to notice that the open leaf of $\mcf$ on $U$, namely the set $U \setminus (\{x=0\} \cup \{y = 0\})$,
clearly intersects the boundary of $U$ transversely. Similarly, the coordinate axes also have transverse intersection
with the boundary of $U$. Naturally the coordinate axes are invariant by both $X$ and $Y$, irrespectively of whether
or not one (or both) of these vector fields vanish identically over the axes in question. The present situation therefore
contrasts with the example of non-Hausdorff leaf space provided by Proposition~\ref{prop_example}.

\begin{proof}[Proof of Theorem~\ref{FinalTheoremCommutingVF}]
There are a few different cases that need to be considered in the proof according to whether or not
$g(0) =0$ and $ab =0$. Let us first consider the case where $ab \ne 0$ and $g(0) \ne 0$. This is the most
representative situation in the sense that the corresponding discussion applies, with very minor modifications,
to the remaining cases.

Assume then $ab \neq 0$ and $g(0) \neq 0$. Since $g(0) \neq 0$ we can assume without loss of generality that
$g(0)=1$. In view of Proposition \ref{Prop_leaf_Sxy}, to prove the theorem, it suffices to check that the leaf
space associated with $S_{x,y}$ remains Hausdorff when the leaf space on $S_0 \cup S_x \cup S_y$ is added.
We are going to begin by looking to the leaves along $S_0 \cup S_x \cup S_y$. So, owing to
Table~\ref{table: spaceofleaves}, there follows that the leaf space of $\overline{\mcf}$ restricted to $S_y$
(resp. $S_x$) is a cylinder. In fact, this space is given by the quotient of $\C$ by the translation
$s \mapsto s + 2\pi i /am$ (resp. $s \mapsto s - 2\pi i /bm$). We need to check how this cylinder glues
together with the fiber above the origin.

\vspace{0.2cm}

\noindent {\it Claim}. The leaf space of $\overline{\mcf}$ associated with its restriction to
$S_0 \cup S_y$ (resp. $S_0 \cup S_x$) is a disc and, hence, Hausdorff.
In particular, the leaf space associated with the restriction of $\overline{\mcf}$ to
$S_0 \cup S_x \cup S_y$
is itself the union of two discs with a single transverse intersection. In particular this space is Hausdorff.
\vspace{0.2cm}

\noindent {\it Proof of the claim}. As previously said, the leaf space of $\overline{\mcf}$ associated with $S_y$ (resp.
$S_x$) is a cylinder. To prove that the leaf space of $\overline{\mcf}$ associated with $S_0 \cup S_y$
(resp. $S_0 \cup S_x$) is a disc it suffices to prove that the leaf space associated with this space can
be obtained by adjunction of a point to one of the ends of the cylinder in question. Clearly it suffices to deal
with the leaf space associated with $S_0 \cup S_y$ since the other case is analogous.

Let $L_1$ and $L_2$ be two distinct leaves over $S_0 \cup S_y$. First, we want to prove the existence of neighborhoods $U_1$ and
$U_2$ of $L_1$ and $L_2$, respectively, such that $U_1 \cap U_2 = \emptyset$. The existence of these neighborhoods immediately
follows in the case where both $L_1$ and $L_2$ are different from $S_0$. Let us then assume that $L_2$ coincides with $S_0$.

Consider the leaf $L_1$. This leaf intersect the fiber $\Sigma_{y_0}$ above $y_0$ for every $y_0 \in \C^{\ast}$. The connected
components of the intersection of $L_1$ with $\Sigma_{y_0}$ takes on the form $\{(t, s_1 + 2k\pi i/(am), y_0): t \in \R\, \,
k \in \Z \} \subseteq L_1$, for some fixed $s_1 \in \C$.

Next, note that the vector field $\partial /\partial s + amy \partial /\partial y$ is tangent to the leaves of $\overline{\mcf}$
contained $S_0 \cup S_y$ since it coincides with the restriction of $\overline{Y}$ to
$\{x=0\}$. Furthermore, the integral curves of this
vector field satisfy
\[
s = s_{y_0} + \frac{1}{am} \left( \ln y - \ln y_0 \right) \, ,
\]
where $(s_{y_0},y_0)$ corresponds to the initial condition of the associated differential equation. If $s = u+iv$, with
$u,v \in \R$ and $y = |y| e^{i\theta}$, the equation above implies that $u, \, v, \, |y|$ and $\theta$ satisfy
\begin{align}
u &= u_{y_0} + \frac{1}{am} \left(\ln |y| - \ln |y_0| \right) \label{eq_Re(s)} \, ,\\
v &= v_{y_0} + \frac{1}{am} \left(\theta - \theta_0 \right) .
\end{align}

Fix then a point $(t_0, s_0, y_0) \in L_1$ and recall that the points $(t,s,y) \in L_1$ satisfy Equation~(\ref{eq_Re(s)}),
where $u$ stands for the real part of $s$ (see Figure~\ref{Figure2}). Let us then consider the following two open sets on
$\Sigma_{y_0}$
\begin{align}\label{eq_neighborhoods}
V_1 &= \{(t,s) \in \C^2 : {\rm Re}(s) < {\rm Re}(s_0) + 1 \} \nonumber \\
V_2 &= \{(t,s) \in \C^2 : {\rm Re}(s) > {\rm Re}(s_0) + 1 \} \, .
\end{align}
Clearly $V_1 \cap V_2 = \emptyset$. In particular the
saturated sets of $V_1 \times \{y_0\}$ and of $V_2 \times \{y_0\}$ by $\overline{\mcf}$ have empty intersection
as well. Denoting respectively by $U_1$ and $U_2$ these saturated sets, there follows that:
\begin{itemize}
  \item $U_1$ is a neighborhood of $L_1$ in the corresponding leaf space;
  \item $U_2 \cup S_0$ is an open neighborhood of $L_2=S_0$ in the same leaf space (i.e. $U_2$ itself is a punctured
  neighborhood of $L_2=S_0$).
\end{itemize}
Since $U_1 \cap (U_2 \cup S_0) = \emptyset$, we conclude that the leaf space of $\overline{\mcf}$ restricted
to $S_0 \cup S_y$ (or, analogously, to $S_0 \cup S_x$) is Hausdorff.

\begin{figure}[hbtp]
\centering
\includegraphics[scale=0.8]{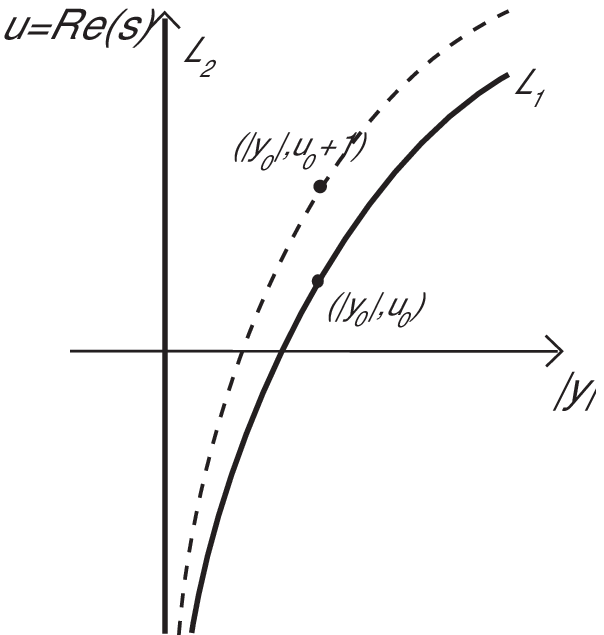}
\caption{}
\label{Figure2}
\end{figure}

To finish the proof of the claim, it only remains to check that the leaf $S_0$ corresponds to one of the ends of the
cylinder associated to the space of leaves of $S_y$. For this recall that the space of leaves associated
with $S_y$ is identified with
\[
\left\{ s= u +iv: u \in \R, \, v \in \left[0, \frac{2\pi}{am} \right] \right\} / \left( u+0i \sim u+2\pi i/(am) \right) \, .
\]
Fix $(s_0,y_0) \in \C \times \{y: (0,y) \in U\}$ with $y_0 \ne 0$ and consider the leaf $L$ of $\overline{\mcf}$ restricted to
$\{x=0\}$ passing through $(0,s_0,y_0)$. As previously shown, the real part of $s$ and the absolute value of $y$ over the mentioned
leaf satisfy
\[
{\rm Re}(s) = {\rm Re}(s_0) + \frac{1}{am} \left(\ln |y| - \ln |y_0| \right) \, .
\]
Furthermore, denoting by $U_1$ and $U_2$ the saturated of $V_1 \times \{y_0\}$ and $V_2 \times \{y_0\}$ with $V_1$ and $V_2$ as
in~(\ref{eq_neighborhoods}), we have that $U_1$ is a neighborhood of $L$ while $U_2 \cup S_0$ is a neighborhood
of $S_0$. Since, by construction, they are $\overline{\mcf}$-invariant, they induce neighborhoods separating
$L$ and $S_0$ in the corresponding leaf space. Finally, note that
all leaves passing through a point $(0,s, y_0)$ such that
the real part of $s$ is less than the real part of $s_0$ remain ``far from'' $S_0$. In turn, as far as we make the real part of
$s$ go to $+\infty$, the neighborhood $U_2 \cup S_0$ of $S_0$ separating $S_0$ from $L_1$ must be made smaller. In other words,
the point associated to $L$ in the leaf space of $\overline{\mcf}$ goes to the point associated to $S_0$. The point associated to
$S_0$ corresponds then to the ``upper end'' of the cylinder. The result follows.

Now it is clear from what precedes that the leaf space over $S_0 \cup S_y \cup S_y$ is the union of two discs
whose intersection is transverse and reduced to the origin. In particular it is Hausdorff. The claim is proved.
\qed

\smallbreak

To finish the proof of the theorem in the case $ab\neq 0$ and $g(0) \neq 0$, it remains to prove that the leaves
contained in the coordinate hyperplanes can be separated from leaves contained in $S_{x,y}$. Here recall that we have
normalized $Y$ to have $g(0) =1$. We are going to prove
that leaves contained in $S_x$ can be separated from leaves contained in $S_{x,y}$. In order to do that we are going
to restrict ourselves to the intersection of the foliation with the hyperplane $\{x=1\}$ (that is parallel to the
previously used hyperplane $\{x=0\} = S_0 \cup S_y$) and to consider the vector field defined as
\begin{eqnarray*}
\overline{Z} & = & g(x^ny^m) \left( -bm + x^ay^b f(x^ny^m) \right)\overline{X} - m x^a y^b \overline{Y} \medskip\\
& = & g(x^ny^m) \left(-bm + x^ay^bf(x^ny^m)\right)\frac{\partial}{\partial t} - m x^ay^b \frac{\partial}{\partial s}
- m x^{a}y^{b+1} g(x^ny^m)\frac{\partial}{\partial y}
\end{eqnarray*}
that is tangent to the foliation $\overline{\mcf}$ and leaves the hyperplane $\{x=1\}$ invariant. Note that every leaf
on $S_x$ intersects $\{x=1\}$ transversely and, in particular, $\Sigma_{(1,0)}$, the fiber above the point $(1,0)$. Also,
every leaf on $S_{x,y}$ intersects $\{x=1\}$ transversely as also $\Sigma_{(1,1)}$, the fiber above the point $(1,1)$.
The proof then follows by repeating the calculations we made for $\{x=0\} = S_0 \cup S_y$ with the vector field
$\overline{Z}$. Indeed, the effect of the monodromy of $\overline{\mcf}$ does not impact the argument since it acts only
on the imaginary part of $s$ (see Table~\ref{table: spaceofleaves})
while the construction of the mentioned foliated neighborhoods relies only on the real part
of~$s$. We have then proved that the entire leaf space of $\overline{\mcf}$ is Hausdorff.

Let us now consider the case where $ab \ne 0$ and $g(0)=0$. Let us check that the leaf space associated with $S_{x,y}$ remains
Hausdorff when the leaf space on $S_0 \cup S_x \cup S_y$ is added in this case. It becomes clear from Table~\ref{table: spaceofleaves}
that the leaf space associated with $S_0 \cup S_x \cup S_y$ is itself the union of two discs with a single transverse intersection
and hence Hausdorff. Indeed, the leaves of $\overline{\mcf}$ contained in $S_0 \cup S_x \cup S_y$ are all parallel to the
$(s,t)$-coordinate plane.

It remains to prove that leaves contained in $S_0 \cup S_x \cup S_y$ can be separated from leaves contained in $S_{x,y}$.
To prove that we can separate, for example, leaves on $S_x$ from leaves on $S_{x,y}$, we have just need to consider again
the intersection of the foliation $\overline{\mcf}$ with the hyperplane $\{x=1\}$ and the vector field $\overline{Z}$ previously
defined. Recall that $\overline{Z}$ is tangent to $\overline{\mcf}$ and leaves $\{x=1\}$ invariant. The difference from the
previous case is that $\Sigma_{(1,0)}$ is itself a leaf of the foliation. Apart from that, all the calculation can be made
as in the previous case.

Finally, the case where $ab=0$ can similarly be treated and will thus be left to the reader. The proof of the theorem is now
complete.
\end{proof}

\begin{proof}[Proof of Theorem~A]
The theorem follows at once from the combination of Proposition~\ref{Maximal-X-Y-TheoremA} and
Theorem~\ref{FinalTheoremCommutingVF}.
\end{proof}

\bigskip

\noindent {\bf Acknowledgements}.

We are very indebted to the referee whose
comments, suggestions, and critics led to corrections and significant improvements on the first version of this paper.
We are also grateful to A. Glutsyuk for his interest in our work.

The first author was partially supported by FCT (Portugal) through the sabbatical grant SFRH/BSAB/135549/2018
and through CMAT (UID/MAT/00013/2013). The second author was partially supported by CIMI - Labex Toulouse - through
the grant ``Dynamics of modular vector fields, Dwork family, and applications''.
The third author was partially supported by CMUP (UID/MAT/00144/2013), which is funded by FCT (Portugal) with national (MEC)
and European structural funds through the programs FEDER, under the partnership agreement PT2020. Finally all the
three authors benefited from CNRS (France) support through the PICS project ``Dynamics of Complex ODEs and Geometry''.

\bigskip

\bigskip

\begin{flushleft}
{\sc Ana Cristina Ferreira} \\
Centro de Matem\'atica da Universidade do Minho, \\
Campus de Gualtar, \\
4710-057 Braga, Portugal\\
anaferreira@math.uminho.pt \\

\end{flushleft}

\bigskip

\begin{flushleft}
{\sc Julio Rebelo} \\
Institut de Math\'ematiques de Toulouse ; UMR 5219\\
Universit\'e de Toulouse\\
118 Route de Narbonne\\
F-31062 Toulouse, FRANCE.\\
rebelo@math.univ-toulouse.fr

\end{flushleft}

\bigskip

\begin{flushleft}
{\sc Helena Reis} \\
Centro de Matem\'atica da Universidade do Porto, \\
Faculdade de Economia da Universidade do Porto, \\
Portugal\\
hreis@fep.up.pt \\

\end{flushleft}

\end{document}